\documentclass[reqno, 11pt]{amsart}
\usepackage{amssymb, amsmath, amsthm, mathrsfs, amsfonts, latexsym, color}
\usepackage[margin=1.1in]{geometry}

\newtheorem{theorem}{Theorem}[section]
\newtheorem{lemma}[theorem]{Lemma}

\newtheorem{proposition}[theorem]{Proposition}
\newtheorem{corollary}[theorem]{Corollary}
\newtheorem{remark}[theorem]{Remark}

\theoremstyle{definition}

\newtheorem{assumption}[theorem]{Assumption}

\numberwithin{equation}{section}

\def\R{{\mathbb R}}

\def\C{{\mathbb C}}
\def\eps{\varepsilon}
\def\Re {{\rm Re}\,}

\def\E{{\mathbb E}}
\def\P{{\mathbb P}}
\def\Z{{\mathbb Z}}

\begin{document}

\title[Chung-type LIL and exact moduli of continuity]
{Chung-type law of the iterated logarithm and 
exact moduli of continuity for a class of
anisotropic Gaussian random fields}
\author{Cheuk Yin Lee}
\address{Institut de math\'ematiques, \'Ecole polytechnique f\'ed\'erale de Lausanne,
Station 8, CH-1015 Lausanne, Switzerland}
\email{cheuk.lee@epfl.ch}
\author{Yimin Xiao}
\address{Department of Statistics and Probability, Michigan State University, East Lansing, MI 48824, United States}
\email{xiao@stt.msu.edu}



\begin{abstract}
We establish a Chung-type law of the iterated logarithm and
the exact local and uniform moduli of continuity 
%and the exact uniform modulus of continuity 
for a large class of anisotropic Gaussian random fields with a harmonizable-type integral representation 
and the property of strong local nondeterminism. Compared with the existing results in the literature, 
our results do not require the assumption of stationary increments and provide more precise upper 
and lower bounds for the limiting constants. The results are applicable to the solutions of a class 
of linear stochastic partial differential equations driven by a fractional-colored Gaussian noise, 
including the stochastic heat equation.
\end{abstract}

\maketitle

\section{Introduction}

%Anisotropic Gaussian random fields arise not only in probability theory but also in 
%various applications such as image processing, hydrology, geostatistics and spatial 
%statistics, because many data sets from these areas have an anisotropic nature in 
%the sense that they have different geometric and statistical features along different 
%directions; see e.g.~\cite{BMBS06, BE03, DH99}.
%Several classes of anisotropic Gaussian random fields have been introduced and 
%studied for theoretical and application purposes. 
%Examples include fractional Brownian sheets \cite{K96}, anisotropic Gaussian
%random fields with stationary increments \cite{BE03} and 
%operator-scaling Gaussian (and stable) random fields \cite{BMS07}.
%Another important example of anisotropic Gaussian random fields is the solutions
%of stochastic partial differential equations (SPDEs); see e.g.~ \cite{B12,D99,HSWX,
%MT02,OZ01,TX17}.

The purpose of this paper is to establish a general framework that is useful for studying 
the regularity properties of sample functions of anisotropic Gaussian random fields and 
can be directly applied to the solutions of linear SPDEs. This is mainly motivated by 
\cite{DMX17} and \cite{X09}. 
%to study the local and uniform regularity properties of anisotropic Gaussian random 
%fields under a general framework proposed by Dalang, Mueller and Xiao \cite{DMX17}
%In particular, we establish Chung-type law of the iterated logarithm, the exact local 
%and uniform moduli of continuity 
We consider a class of Gaussian random fields $\{v(x), x \in \R^k\}$ that satisfy Assumption 
2.1 in  \cite{DMX17} (see Assumption \ref{a1} below) and the property of strong local 
nondeterminism, or strong LND for short, with respect to an anisotropic metric (see 
Assumption \ref{a2} below). For these Gaussian random fields, we prove some limit 
theorems that provide precise information about the oscillation behavior of the sample 
function $x \mapsto v(x)$. 

The main results of this paper are as follows. We prove a Chung-type law of the iterated 
logarithm (LIL) in Theorem \ref{Thm:ChungLIL},  the exact local and uniform moduli of 
continuity in Theorems \ref{Thm:LIL} and \ref{Thm:MC}, respectively. Our strategy is 
to first prove a zero--one law for each of the limit theorems (see Lemma \ref{0-1law}), 
showing that the limit is equal to a constant in $[0, \infty]$ almost surely. Then, we prove 
that the constant is in fact positive and finite by establishing a finite upper bound and 
positive lower bound for the limit, and therefore, the corresponding modulus function 
in the limit theorem is sharp. We give an application of the main results to the solutions 
of a class of linear SPDEs
\[ 
\frac{\partial}{\partial t} u(t, x) = \mathscr{L} u(t, x) + \dot W(t, x) 
\]
driven by a fractional-colored Gaussian noise, including the stochastic heat equation 
\cite{BT08,HSWX}. It is also a notable result of this paper that $u(t, x)$ satisfies the 
strong LND property (see Lemma \ref{Lem:SHE-a2}), which strengthens a result of 
\cite{HSWX}. 

In general, there are different ways to describe the sample path variation of random 
fields. The Chung-type LIL characterizes the lower envelope ($\liminf$) for the local 
oscillations of the sample functions at a fixed point. The local modulus of continuity at 
a fixed point is, for many Gaussian random fields, given by the ordinary Khinchin-type 
LIL, which complements the Chung-type LIL by characterizing the upper envelope 
($\limsup$) for the local oscillations at a fixed point. On the other hand, the uniform 
modulus of continuity specifies the maximum oscillation of the sample functions over 
certain sets such as a compact interval.

The Chung-type LIL for a class of isotropic and anisotropic Gaussian random fields 
with stationary increments has been studied by Li and Shao \cite{LS01} (see also 
\cite{MR95,X97}) and Luan and Xiao \cite{LX10}. The exact local and uniform moduli 
of continuity for a class of anisotropic Gaussian random fields have been studied by 
Meerschaert et al.~\cite{MWX13}. The novelty of the present paper is a general 
framework based on a harmonizable-type representation and the strong LND 
property of a Gaussian random field that may not have stationary increments, which 
is employed to extend and improve some of the results of \cite{LS01, LX10, MWX13}, and can be directly applied to the solutions of SPDEs. 
In particular, with a harmonizable-type 
representation, we are able to decompose the random field and create independence, 
making it possible to establish general zero--one laws (Lemma \ref{0-1law}) which can 
be strengthened to prove the Chung-type LIL as well as the exact local and uniform 
moduli of continuity. The independence structure from the harmonizable-type representation 
also allows the use of the second Borel--Cantelli lemma, which facilitates a simple proof 
of one of the bounds for the Chung-type LIL and the exact local modulus of continuity.

%A harmonizable-type representation is a stochastic integral representation that is similar 
%to the spectral representation of a stationary Gaussian process (or one with stationary 
%increments), but the former is more general because it does not require the process to 
%be stationary or have stationary increments. Dalang et al.~\cite{DMX17, DLMX} have 
%used harmonizable-type representations to study the polarity of points and existence of 
%multiple points in critical dimensions for Gaussian random fields and systems of linear 
%SPDEs including stochastic heat and wave equations. Harmonizable representations 
%for SPDEs also appear in \cite{B12}.

%Besides, the strong LND property is a useful tool for dealing with complex 
%dependence structures of Gaussian random fields. It has found various 
%applications in studying probabilistic, analytic and fractal properties of 
%Gaussian random fields. We refer to the survey articles of Xiao \cite{X08, X09} 
%for details. In the present paper, we use strong LND to establish optimal bounds 
%for the small ball probability in Proposition \ref{Prop:SB}. It is well known that 
%the small ball probability is a key step in proving Chung's LIL 
%\cite{MR95, LS01}. We also use strong LND to prove an optimal lower bound 
%for the exact uniform modulus of continuity in Theorem \ref{Thm:MC}.

Let us summarize the major differences and improvements in our results compared 
to the existing results in the literature. The Chung-type LIL results in \cite{LS01}, 
\cite{MR95}, \cite{X97} and \cite{LX10} were proved for Gaussian random fields 
%(see also  \cite{X97}), while \cite{LX10}  concerns with anisotropic Gaussian random 
%fields with strong LND. All these references assume that the random field has 
with stationary increments, meaning that for any $h \in \R^k$,
\[ 
\{v(x+h) - v(h), x \in \R^k \} \overset{d}{=} \{ v(x) - v(0), x \in \R^k \},  
\]
and, in particular, it is enough for them to consider the Chung-type LIL at the origin.
Our Theorem \ref{Thm:ChungLIL} applies to a wider class of Gaussian random fields 
that may not necessarily have stationary increments, and we prove a Chung-type LIL 
at any fixed point $x_0$. 
%Examples of such Gaussian random fields are given in Sections 7 and 8.
%For example, 
%in Section 8 of this paper, we use the integral representation to construct a class 
%of strongly LND anisotropic Gaussian random fields that do not have stationary 
%increments. 
%In our Proposition \ref{Prop:SB}, we generalize the small ball estimates 
%at the origin in \cite[Theorem 5.1]{X09} %[Lemma 2.2]{LX10} 
%to small ball estimates centered at the value $v(x_0)$ for any fixed point $x_0$.
Moreover, our Theorem \ref{Thm:ChungLIL} gives explicit upper and lower bounds 
for the constant in Chung's LIL in terms of the constants that appear in the 
small ball probability estimates. This implies that the limiting constant in Chung's 
LIL is given in terms of the small ball constant provided it exists, see \eqref{Eq:SBcon} 
below. We remark that the connection between the bounds on the limiting constant 
in Chung's LIL and the small ball estimates is also given in Theorem 7.1 of \cite{LS01}, 
but not explicitly stated in Theorem 1.1 of \cite{LX10}.

For exact local and uniform moduli of continuity of Gaussian processes with stationary 
increments, some general theory has been established by Marcus and Rosen \cite{MR}. 
Also, Meerschaert et al.~\cite{MWX13} have used the sectorial LND property and 
the Fernique-type inequalities to prove the exact uniform modulus of continuity of 
anisotropic Gaussian random fields. Especially,  \cite{MWX13} provides an effective 
way to prove the lower bound for the uniform modulus of continuity, which is usually 
a more difficult task than proving the upper bound. Our Theorems \ref{Thm:LIL} and 
\ref{Thm:MC}  improve  the results in \cite{MWX13}. 
We prove exact local and uniform moduli of continuity under two metrics 
respectively: one is the canonical metric $d$ defined in \eqref{Def:d}, 
and the other one is the metric  $\Delta$ defined in \eqref{Def:Delta}, 
which is comparable to $d$ under Assumptions \ref{a1} and \ref{a3}.
For the local modulus of continuity in Theorem \ref{Thm:LIL}, under the canonical 
metric $d$, we are able to prove that the exact constant in the LIL is $\sqrt 2$.
This is an improvement to Theorem 5.6 of \cite{MWX13}, which only shows that 
the constant is at least $\sqrt 2$ (see Remark \ref{Rmk:LIL} below).
We achieve this sharper result by using a tail probability estimate due to 
Talagrand \cite{T94}, which is stated in Lemma \ref{T94} below.

For the uniform modulus of continuity, the strong LND assumption in our Theorem 
\ref{Thm:MC} is stronger than the condition in Theorem 4.1 of \cite{MWX13}, but 
we obtain better upper and lower bounds for the limiting constant (see Remark \ref{Rmk:MC} 
below). Our approach is to start with a crude upper bound and then optimize it using 
an approximation argument based on anisotropic lattice points. We also refine the the 
proof in \cite{MWX13} based on the strong LND property and a conditioning argument 
to get a sharper lower bound.

The rest of the paper is organized as follows. In Section 2, we state the assumptions 
for the Gaussian random fields to be considered in this paper and give some remarks 
about the assumptions. We also briefly discuss the difference in the LND properties 
between stochastic heat and wave equations. In Section 3, we prove zero--one laws 
which will be useful for establishing the Chung-type LIL and the exact local and uniform 
moduli of continuity. In Section 4, we establish small ball probability estimates and 
Chung's LIL. In Sections 5 and 6, we prove the exact local and uniform moduli of continuity,
respectively. In Section 7, we consider as an application a class of linear  SPDEs driven 
by a fractional-colored Gaussian noise \cite{BT08,HSWX}. We establish harmonizable-type 
representations and strong LND property for the solutions, and apply our results to obtain 
Chung's LIL and exact local and uniform moduli of continuity. These results improve 
significantly those in \cite{HSWX,TX17}. Finally, in Section 8, we provide another example 
of anisotropic Gaussian random fields that do not have stationary increments and satisfy 
Assumptions \ref{a1} and \ref{a2} of the present paper.

\section{Assumptions}

Consider a real-valued continuous centered Gaussian random field $v = \{ v(x), x \in \R^k \}$. 
Let $T$ be a compact rectangle in $\R^k$. We introduce some assumptions for $v$. Notice 
that Assumption \ref{a1}  is from \cite{DMX17} and Assumption \ref{a2} is from \cite{X09}.

\begin{assumption}\label{a1}
There exists a centered Gaussian random field $\{ v(A, x) , A \in
\mathscr{B}(\mathbb{R}_+), x \in T \}$, where $\mathscr{B}(\mathbb{R}_+)$ is the
Borel $\sigma$-algebra on $\mathbb{R}_+ :=[0, \infty)$, such that  the following 
properties hold:

(a) For every $x \in T$, $A \mapsto v(A, x)$ is an independently scattered Gaussian 
noise such that $v(\mathbb{R}_+, x) = v(x)$ and 
the processes $v(A, \cdot)$ and $v(B, \cdot)$ are independent whenever $A$ and $B$ 
are disjoint.

(b) There exist constants $c_0 > 0$, $ a_0 \ge 0$, and $\gamma_j > 0$, $j = 1, \dots, k$, 
such that for all $a_0 \le a < b \le \infty$
and $x, y \in T$,
\begin{align}\label{Eq:a1-1}
\big\| v([a, b), x) - v(x) - v([a, b), y) + v(y) \big\|_{L^2}
\le c_0 \bigg( \sum_{j=1}^k a^{\gamma_j}|x_j - y_j| + b^{-1} \bigg)
\end{align}
and
\begin{align}\label{Eq:a1-2}
\big\| v([0, a_0), x) - v([0, a_0), y) \big\|_{L^2} \le c_0 \sum_{j=1}^k |x_j - y_j|.
\end{align}
In the above, $\|X\|_{L^2} := [\E(X^2)]^{1/2}$ for a random variable $X$.
\end{assumption}

Define $\alpha_j$ $(j = 1, \ldots, k$) by the relation $\gamma_j = \alpha_j^{-1} - 1$, that is,
$\alpha_j = (\gamma_j + 1)^{-1}$. Note that $0 < \alpha_j < 1$.
The parameters $\alpha_j$ characterize the H\"older regularity of $v$ 
(see Lemma \ref{Lem:DMX17} below).
Let $Q = \sum_{j = 1}^k \alpha_j^{-1}$
and define the metric $\Delta$ by
\begin{equation}\label{Def:Delta}
\Delta(x, y) := \sum_{j=1}^k |x_j - y_j|^{\alpha_j}, \quad x, y \in \R^k.
\end{equation}
We will also use the canonical metric $d = d_v$ associated with $v$. 
It is defined by
\begin{equation}\label{Def:d}
d(x, y) = d_v(x, y) := \big\|v(x) - v(y) \big\|_{L^2}, \quad x, y \in \R^k.
\end{equation}

\begin{assumption}\label{a2}
There exists a constant $c_2 > 0$
such that for all integers $n \ge 1$,
for all $x, x^1, \dots, x^n \in T$,
\[ \mathrm{Var} \big(v(x) | v(x^1), \dots, v(x^n) \big) \ge 
c_2 \min_{0\le i \le n} \Delta^2(x, x^i), \]
where $x^0 = 0$.
\end{assumption}

\begin{assumption}\label{a3}
There exists a constant $c_3 > 0$ such that for all $x, y \in T$,
\[  \big\|v(x) - v(y)\big\|_{L^2} \ge c_3 \Delta(x, y). \]
\end{assumption}

\medskip

The following are some remarks about these assumptions. Assumption \ref{a1} above 
is the same as Assumption 2.1 in \cite{DMX17} and \cite{DLMX}, and is satisfied by 
many Gaussian random fields that have a spectral or harmonizable-type representation, 
or more generally, a stochastic integral representation. For example, it is shown in 
\cite{DMX17} that the solutions of linear stochastic heat and wave equations admit 
harmonizable-type representations and satisfy Assumption \ref{a1}. 
The same is true for fractional Brownian sheets \cite{DLMX}.  
%Harmonizable representations for SPDEs also appear in \cite{B12}.
Assumption \ref{a1} implies an upper bound for the increments of $v$ in $L^2$-norm in 
terms of the metric $\Delta$:

\begin{lemma}\label{Lem:DMX17}
Under Assumption \ref{a1}, there exist constants $\eps_1 > 0$ and $c_1$ 
such that for all $x, y \in T$ with $\Delta(x, y) \le \eps_1$,
\begin{equation}\label{UB:d}
\|v(x) - v(y) \|_{L^2} \le c_1 \Delta(x, y).
\end{equation}
\end{lemma}

\begin{proof}
This is a consequence of Proposition 2.2 of \cite{DMX17} where 
$\eps_1 = \min\{a_0^{-1}, 1\}$ and $c_1 = 4c_0$.
\end{proof}

Assumption \ref{a2} is known as the property of strong local nondeterminism (strong LND) 
with respect to the metric $\Delta$, which has found various applications in studying probabilistic, 
analytic and fractal properties of Gaussian random fields (cf.~\cite{X08, X09}).
For Gaussian random fields with stationary
increments, \cite{LX12} provides sufficient conditions in terms of their spectral measures for 
them to have the property of strong LND. In Sections 7 and 8, we will show that the solutions
of a class of linear SPDEs driven by a fractional-colored Gaussian noise and a 
class of Gaussian random fields with non-stationary increments also have the property of 
strong LND. 
 
Assumption \ref{a2} implies the lower bound in Assumption \ref{a3} if 
$T \subset \R^k\backslash \{0\}$ is compact: for all $x, y \in T \subset \R^k\backslash \{0\}$,
\begin{equation}\label{LB:d}
\|v(x) - v(y)\|_{L^2} \ge \sqrt{c_2} \Delta(x, y).
\end{equation}

The strong LND property (Assumption \ref{a2}) will be used to prove optimal bounds for 
the small ball probability, which is the main ingredient of the proof of Chung's LIL (Theorem 
\ref{Thm:ChungLIL}). This property will also be needed in the proof of the lower bound 
for the exact uniform modulus of continuity (Theorem \ref{Thm:MC}). Assumption \ref{a3} 
is much weaker than Assumption \ref{a2}. To establish the exact local modulus of continuity, 
or the ordinary LIL (Theorem \ref{Thm:LIL}), we will use Assumption \ref{a3}, 
but not Assumption \ref{a2}.

In Lemma \ref{Lem:SHE-a2} of this paper, we will prove that the solutions of a class of 
linear SPDEs with a fractional-colored Gaussian noise, including the stochastic 
heat equation, satisfy the strong LND property. We remark that, on the other hand, the solution 
of the linear stochastic wave equation does not satisfy the strong LND property, but satisfies 
a different form of LND \cite{LX19, L21+}. For this reason, the Chung-type LIL for the stochastic 
wave equation has a different form than the stochastic heat equation; see \cite{L21}. The 
uniform modulus of continuity for the stochastic wave equation is also established in 
\cite{LX19, L21+}. 
For the local modulus of continuity, our Theorem \ref{Thm:LIL} still applies to 
the stochastic wave equation because it does not require the strong LND property.
It also applies to fractional Brownian sheets.

\section{Zero--one laws}

In Lemma \ref{0-1law} below, we establish zero--one laws 
for the Chung-type LIL and the local and uniform moduli of continuity,
showing that the limit in each of these laws is equal to a constant almost surely.
At this stage, we do not rule out the possibility that the constant could be zero 
or infinity. Later in our main theorems, we will strengthen these zero--one laws
and prove that the limiting constants are indeed positive and finite. 
%establish finite upper bounds and positive lower bounds for the limits.

\begin{lemma}\label{0-1law}
The following statements hold under Assumption \ref{a1}. %and \ref{a3}.
\begin{enumerate}
\item[(i)] For any fixed $x_0 \in T$, 
there exists a constant $0 \le \kappa_1 \le \infty$ which may depend on $x_0$ such that
\begin{equation}\label{0-1law:1}
\liminf_{r \to 0+} \sup_{x \in T:\, \Delta(x, x_0) \le r}
\frac{|v(x) - v(x_0)|}{r (\log\log(1/r))^{-1/Q}} 
= \kappa_1 \quad \text{a.s.}
\end{equation}
\item[(ii)] For any fixed $x_0 \in T$, there exists a constant $0 \le \kappa_2 \le \infty$ 
which may depend on $x_0$ such that
\begin{equation}\label{0-1law:2}
\lim_{r \to 0+} \sup_{x \in T:\, 0 < \Delta(x, x_0) \le r}
\frac{|v(x) - v(x_0)|}{\Delta(x, x_0) \sqrt{\log\log(\Delta(x, x_0)^{-1})}} = \kappa_2
\quad \text{a.s.}
\end{equation}
\item[(iii)] There exists a constant $0 \le \kappa_3 \le \infty$ such that
\begin{equation}\label{0-1law:3}
\lim_{r \to 0+} \sup_{x, y \in T:\, 0 < \Delta(x, y) \le r}
\frac{|v(x) - v(y)|}{\Delta(x, y)\sqrt{\log(\Delta(x, y)^{-1})}} = \kappa_3 \quad \text{a.s.}
\end{equation}
\end{enumerate}
Moreover, under Assumptions \ref{a1} and \ref{a3}, 
%\footnote{\textcolor{blue} {The first part does not seem to need  Assumption 2.3. So I have moved it here.} }
\eqref{0-1law:2} and \eqref{0-1law:3} 
also hold when $\Delta$ is replaced by the canonical metric $d$, 
with possibly different constants.
\end{lemma}

\begin{proof}
By Assumption \ref{a1}, $v(x)$ can be represented as the infinite sum
\begin{equation}\label{v-series}
v(x) = \sum_{n=0}^\infty v_n(x),
\end{equation}
where $v_n(x) = v([n, n+1), x)$ and $v_n = \{ v_n(x), x \in T\}$ ($n = 0, 1, \dots$) 
is a sequence of independent Gaussian random fields. 
Let $\mathscr{F}_n$ be the $\sigma$-algebra generated by the processes
$\{v_m, m \ge n\}$ and the null events, and let
$\mathscr{F}_\infty = \bigcap_{n=0}^\infty \mathscr{F}_n$ be the $\sigma$-algebra 
of all tail events.
By Kolmogorov's zero--one law, $\P(A) = 0$ or $1$ for $A \in \mathscr{F}_\infty$.

To prove (i), we will show that for any fixed $x_0 \in T$, the random variable
\[ X := \liminf_{r \to 0+} \sup_{x \in T:\, \Delta(x, x_0) \le r}
\frac{|v(x) - v(x_0)|}{r (\log\log(1/r))^{-1/Q}} \]
is measurable with respect to the $\sigma$-algebra $\mathscr{F}_\infty$.
 For any $n \ge 1$ and $x \in T$, let
\[ Y_n(x) = \sum_{m=0}^{n-1} v_m(x) \quad \text{and} \quad 
Z_n(x) = \sum_{m=n}^\infty v_m(x). \]
Note that $v(x) = Y_n(x) + Z_n(x)$ and $Y_n(x) = v([0, n), x)$. 
Consider $n \ge a_0$, where $a_0$ is the constant in Assumption \ref{a1}.
Then by \eqref{Eq:a1-1} with $a = n$ and $b = \infty$, 
for all $x, y \in T$, we have
\begin{align*}
\| Y_n(x) - Y_n(y) \|_{L^2} \le c_0 \sum_{j=1}^k n^{\gamma_j} |x_j-y_j|.
\end{align*}
Since $Y_n$ is Gaussian, this implies that for any $p \ge 2$, there is a finite constant $C$ 
which depends on $n$ such that for all $x, y \in T$,
\[ \E(|Y_n(x) - Y_n(y)|^p) \le C |x-y|^p. \]
Then, by Kolmogorov's continuity theorem, for any $0 < \beta < 1$, 
with probability one, 
$x \mapsto Y_n(x)$ is $\beta$-H\"older continuous on $T$.
If we choose $\beta$ such that $\max \{\alpha_1, \dots, \alpha_k \} < \beta < 1$, 
then for a.e.~$\omega$, there exists $C = C(\omega, n) < \infty$ such that
for all $x, y \in T$,
\begin{equation}\label{Eq:Y_n}
|Y_n(x) - Y_n(y)| \le C \sum_{j=1}^k |x_j - y_j|^\beta
\end{equation}
This implies that for any $x_0 \in T$,
\begin{equation*}
\lim_{r \to 0+} \sup_{x \in T:\, \Delta(x, x_0) \le r}
\frac{|Y_n(x) - Y_n(x_0)|}{r (\log\log(1/r))^{-1/Q}}
= 0 \quad \text{a.s.}
\end{equation*}
Since $v = Y_n + Z_n$, we have 
\[
X = \liminf_{r \to 0+} \sup_{x \in T:\, \Delta(x, x_0) \le r}
 \frac{|Z_n(x) - Z_n(x_0)|}{r (\log\log(1/r))^{-1/Q}} \quad \text{a.s.}
\]
This means that $X$ is an $\mathscr{F}_n$-measurable random variable, and this is 
true for arbitrary $n \ge a_0$. Therefore, $X$ is $\mathscr{F}_\infty$-measurable.
By Kolmogorov's zero--one law, this implies (i).

For (ii) and (iii), notice that the limits on the left-hand side of 
\eqref{0-1law:2} and \eqref{0-1law:3} both exist by monotonicity.
Moreover, similarly to the above arguments, 
for any $n \ge a_0$, by \eqref{Eq:Y_n}, we have
\begin{equation}\label{LIL:Y_n}
\lim_{r \to 0+} \sup_{x \in T:\, 0 <\Delta(x, x_0) \le r}
\frac{|Y_n(x) - Y_n(x_0)|}{\Delta(x, x_0) \sqrt{\log\log(\Delta(x, x_0)^{-1})}}
= 0 \quad \text{a.s.}
\end{equation}
and
\begin{equation}\label{MC:Y_n}
\lim_{r \to 0+} \sup_{x, y \in T:\,0 < \Delta(x, y) \le r}
\frac{|Y_n(x) - Y_n(y)|}{\Delta(x, y) \sqrt{\log(\Delta(x, y)^{-1})}}
= 0 \quad \text{a.s.}
\end{equation}
It follows that the left-hand side of \eqref{0-1law:2} and \eqref{0-1law:3} are 
$\mathscr{F}_\infty$-measurable random variables and therefore are 
constants a.s.~ by Kolmogorov's zero--one law.

Finally, to see that \eqref{0-1law:2} and \eqref{0-1law:3} also hold when $\Delta$ 
is replaced by $d$, note that for each $n \ge a_0$, by \eqref{Eq:Y_n} and Assumption \ref{a3}, 
for a.e.~$\omega$, there exists $C = C(\omega, n) < \infty$ such that
for all $x, y \in T$ with $d(x, y) \le r$, we have
\[ |Y_n(x) - Y_n(y)| \le C r^{\beta - \alpha^*} d(x, y), \]
where $\alpha^* = \max\{\alpha_1, \dots, \alpha_k\}$.
Therefore, \eqref{LIL:Y_n} and \eqref{MC:Y_n} hold with $\Delta$ being replaced by 
$d$, and the desired result follows from the fact that $v = Y_n + Z_n$ and 
Kolmogorov's zero--one law.
\end{proof}

\section{Chung-type law of the iterated logarithm}

This section is devoted to proving the Chung-type LIL.
It is well known that the small ball probability is a key step in establishing 
Chung's LIL (\cite{MR95, LS01}).
The following lemma is a reformulation of Talagrand's 
lower bound for small ball probabilities of Gaussian processes 
\cite[Lemma 2.2]{T95}. See Ledoux \cite[p.257]{L} for a proof.

\begin{lemma}\label{Lem:T95}
Let $\{ X(t), t \in S \}$ be a separable, real-valued, mean-zero Gaussian process
indexed by a bounded set $S$ with canonical metric 
$d_X(s, t) = \|X(s) - X(t)\|_{L^2}$.
Let $N(S, d_X, \eps)$ denote the  smallest number of $d_X$-balls of 
radius $\eps$ needed to cover the set $S$. 
Suppose there is a decreasing function $\psi : (0, \delta) \to (0, \infty)$ 
such that 
$N(S, d_X, \eps) \le \psi(\eps)$ for all $\eps \in (0, \delta)$ and there are constants 
$a_2 \ge a_1 > 1$ such that for all $\eps \in (0, \delta)$,
\begin{equation}\label{Eq:Covering}
a_1 \psi (\eps) \le \psi (\eps/2) \le a_2 \psi (\eps).
\end{equation}
Then, there is a finite constant $K$ depending only on $a_1$ and $a_2$ such that 
for all $u \in (0, \delta)$,
\begin{equation}\label{Eq:SB1}
\P\bigg\{ \sup_{s, t \in S} |X(s) - X(t)| \le u \bigg\} \ge \exp \left(-K \psi(u) \right).
\end{equation}
\end{lemma}

Recall the metric $\Delta$ %and $d_v$ 
defined in \eqref{Def:Delta}. % and \eqref{Def:d}.
Let $B_\Delta(x, r) = \{ y \in \R^k : \Delta(x, y) \le r \}$ be
the closed $\Delta$-ball centered at $x$ of radius $r$.
In the proposition below, we prove optimal bounds for the small ball probability 
of $v$ around a fixed point $x_0$, which generalizes Theorem 5.1 in \cite{X09}
(or Lemma 2.2 of \cite{LX10}), where the case of $x_0 = 0$ and $r = 1$ was considered.

\begin{proposition}\label{Prop:SB}
Under Assumptions \ref{a1} and \ref{a2}, there exist positive finite constants $C_1, C_2$
and $r_0 > 0$ small such that for any $0 < u < r \le r_0$ and $x_0 \in T$ 
with $B_\Delta(x_0, r) \subset T$, we have
\begin{equation}\label{SB}
\exp\left(-C_1 (r/u)^Q\right) \le 
\P\bigg\{ \sup_{x \in B_\Delta(x_0, r)} |v(x) - v(x_0)| \le u \bigg\}
\le \exp\left(-C_2 (r/u)^Q\right).
\end{equation}
\end{proposition}

\begin{proof}
We first prove the lower bound in \eqref{SB}.
Consider the Gaussian random field $\{ v(x), x \in T \}$ and the canonical metric 
$d_v(x, y) = \|v(x) - v(y)\|_{L^2}$.
By Assumption \ref{a1} and Lemma \ref{Lem:DMX17}, 
we can find some small $r_0 > 0$ such that $d_v(x, y) \le c_1 \Delta(x, y)$
for all $x, y \in T$ with $\Delta(x, y) \le r$ and $0 < r \le r_0$.
This implies $N(B_\Delta(x_0, r), d_v, \eps) \le C_0(r/\eps)^Q$ for all $\eps > 0$ small,
where $C_0$ does not depend on $r$ or $\eps$.
Take $S = B_\Delta(x_0, r)$ and $\psi(\eps) = C_0(r/\eps)^Q$. 
Then $\psi$ satisfies \eqref{Eq:Covering} with $a_1 = a_2 = 2^Q > 1$.
Hence,  the lower bound in \eqref{SB} follows from Lemma \ref{Lem:T95}.

The proof of the upper bound in \eqref{SB} is based on Assumption \ref{a2} and a 
conditioning argument. Suppose $0 < u < r \le r_0$ and $B_\Delta(x_0, r) \subset T$. 
Notice that, for $x_0 = (x_{0,1}, \dots, x_{0,k})$, the rectangle $I :=\prod_{j=1}^k [x_{0,j}, 
\,x_{0,j} + (k^{-1}r)^{1/\alpha_j}]$ is contained in $B_\Delta(x_0, r)$.
For simplicity, we consider the case where $x_0$ lies in the orthant 
$[0, \infty)^k$, so that the interior of $I$ does not contain the origin (otherwise, 
in order to retain this latter property for $I$, we can modify the definition of $I$ by 
using the interval $[x_{0,j} - (k^{-1}r)^{1/\alpha_j}, \,x_{0,j}]$ for $x_{0, j} < 0$
and the rest of the proof is similar).
It suffices to prove that 
\begin{equation}\label{Eq:SB_on_I}
\P\bigg\{ \sup_{x \in I} |v(x) - v(x_0)| \le u \bigg\} \le \exp\left(-C_2 (r/u)^Q\right).
\end{equation}
%where $I$ is the rectangle $\prod_{j=1}^k [x_{0,j}, x_{0,j} + r^{1/\alpha_j}]$ and 
%$x_0 = (x_{0,1}, \dots, x_{0,k})$.
Since $r/u > 1$, we can find an integer $n \ge 2$ 
such that $n-1 < r/u \le n$ (in particular, $n/2 < r/u$).
Divide $I$ into sub-rectangles of side lengths $\big(r/(kn)\big)^{1/\alpha_j}$ ($j = 1, \dots, k$).
The number of sub-rectangles is $N \sim n^Q$.
Let $x_i$ ($1 \le i \le N$) denote the upper-right vertices of the sub-rectangles in any order.
For each $1 \le j \le N$, let
\begin{equation*}
A_j = \bigg\{ \max_{1 \le i \le j} |v(x_i) - v(x_0)| \le u \bigg\}.
\end{equation*}
Then by conditioning,
\begin{equation}\label{Eq:condition}
\P(A_j) = \E\bigg[ {\bf 1}_{A_{j-1}}\, 
\P\Big\{ |v(x_j) - v(x_0)| \le u\, \Big|\, v(x_i) : 0 \le i \le j -1 \Big\} \bigg].
\end{equation}
By Assumption \ref{a2}, the property that the $x_i$'s are separated by a 
$\Delta$-distance of at least $r/(kn)$ and 
that the interior of $I$ does not contain the origin, we have
\begin{align}\label{condvar}
\begin{split}
\mathrm{Var}\big(v(x_j) \big| v(x_i) : 0 \le i \le j-1\big) 
& \ge c_2 \min_{0 \le i \le j-1} \Delta^2(x_j, x_i)  \ge c_2 (r/(kn))^2.
\end{split}
\end{align}
Since the random field $v$ is Gaussian, the conditional distribution of $v(x_j)$ given 
all the $v(x_i)$, with $0 \le i \le j-1$, is a Gaussian distribution with conditional variance 
$\mathrm{Var}(v(x_j)| v(x_i) : 0 \le i \le j-1)$.
Then, \eqref{condvar} and Anderson's inequality \cite[Theorem 2]{A55} imply that
\begin{equation}\label{Eq:cond_prob}
\begin{split}
\P\Big\{ |v(x_j) - v(x_0)| \le u\, \Big|\, v(x_i) : 0 \le i \le j-1 \Big\}
& \le \P\Big\{|Z| \le \frac{u}{\sqrt{c_2}(r/(kn))}\Big\} \le \exp(-C),
\end{split}
\end{equation}
where $Z$ is a standard Gaussian random variable
and the last inequality holds for some constant $C>0$
since $k \le knu/r \le 2k$.
Then, based on \eqref{Eq:condition} and \eqref{Eq:cond_prob}, 
we can use induction to deduce that
\begin{align*}
\P\bigg\{ \max_{1 \le i \le N} |v(x_i) - v(x_0)| \le u \bigg\} = \P(A_N) \le
\exp(-C N).
\end{align*}
Since $N \sim n^Q$ and $n-1< r/u \le n$,
this implies \eqref{Eq:SB_on_I} and completes the proof.
\end{proof}

The next lemma is an isoperimetric inequality for general Gaussian processes.

\begin{lemma}\label{Lem:LT}
\cite[p.302]{LT}
There is a universal constant $K_0$ such that the following statement holds.
Let $S$ be a bounded set and $\{X(s), s \in S\}$ be a separable Gaussian process.
%Let $d_X(s, t) = \|X(s) - X(t)\|_{L^2}$ denote the canonical metric and
%$N(S, d_X, \eps)$ denote the smallest number of $d_X$-balls of radius $\eps$ needed 
%to cover the set $S$. 
Let $D = \sup\{ d_X(s, t) : s, t \in S \}$ be the diameter of $S$ in metric $d_X$.
Then for any $u > 0$,
\begin{equation*}
\P\bigg\{ \sup_{s, t \in S} |X(s) - X(t)| \ge 
K_0\Big( u + \int_0^D \sqrt{\log N(S, d_X, \eps)} d\eps \Big) \bigg\} 
\le \exp\left(-\frac{u^2}{D^2}\right).
\end{equation*}
\end{lemma}

Now, we are ready to prove the Chung-type LIL.

\begin{theorem}\label{Thm:ChungLIL}
Under Assumptions \ref{a1} and \ref{a2}, for any fixed $x_0 \in T$, there exists a 
positive finite constant $\kappa$ which may depend on $x_0$ such that
\begin{equation*}
\liminf_{r \to 0+} \sup_{x \in T:\Delta(x, x_0) \le r}
\frac{|v(x) - v(x_0)|}{r (\log\log(1/r))^{-1/Q}} = \kappa^{1/Q}
\quad \text{a.s.}
\end{equation*}
and $C_2 \le \kappa \le C_1$, where $C_1$ and $C_2$ are the constants in 
Proposition \ref{Prop:SB}. In particular, $\kappa$ coincides with 
the following limit, which is called the small ball constant of $v$ on $\{x \in T: \Delta(x, x_0) \le r\},$ 
if it exists:
\begin{equation}\label{Eq:SBcon}
\kappa = - \lim_{\substack{r \to 0, u/r \to 0}} {\left( \frac{u}{r}\right)}^Q 
\log \P\bigg\{ \sup_{x \in T: \Delta(x, x_0) \le r} |v(x) - v(x_0)| \le u \bigg\}. 
\end{equation}
\end{theorem}

\begin{proof}[Proof of Theorem \ref{Thm:ChungLIL}]
Fix $x_0 \in T$. To simplify notations, define $h(r) := r (\log\log(1/r))^{-1/Q}$ and
\[ 
L(r) := \sup_{x \in T: \Delta(x, x_0) \le r}\frac{|v(x) - v(x_0)|}{h(r)}. 
\]
By Lemma \ref{0-1law}, $\liminf_{r\to 0+}L(r) = \kappa_1$ a.s.~
for some constant $0 \le \kappa_1 \le \infty$. 
To prove the theorem, we will show that
\begin{equation}\label{Chung:LB}
\liminf_{r\to 0+}L(r) \ge C_2^{1/Q} \quad \text{a.s.}
\end{equation}
and 
\begin{equation}\label{Chung:UB}
\liminf_{r\to 0+}L(r) \le C_1^{1/Q} \quad \text{a.s.}
\end{equation}

We first prove the lower bound \eqref{Chung:LB}. Let $a > 1$ be a constant.
For each $n \ge 1$, let $r_n = a^{-n}$.
Consider a constant $K$ such that $0 < K < a^{-1}C_2^{1/Q}$ and 
consider the event
\[
A_n = \bigg\{ \sup_{x \in T: \Delta(x, x_0) \le r_n} |v(x) - v(x_0)| \le K h(r_{n-1}) \bigg\}.
\]
By the small ball probability estimates in Proposition \ref{Prop:SB},
\begin{align*}
\P(A_n) 
&\le \exp\left(-C_2 (aK)^{-Q} \log\log(1/r_{n-1})\right) \\
&= ((n-1)\log a)^{-C_2 (aK)^{-Q}}. 
\end{align*}
Then $\sum_{n=1}^\infty \P(A_n) < \infty$ since $C_2 (aK)^{-Q} > 1$. 
By using the Borel--Cantelli lemma and letting 
$K \uparrow a^{-1}C_2^{1/Q}$ along a rational sequence, we get that
\begin{equation}\label{Chung:LB-n}
\liminf_{n \to \infty} \sup_{x \in T: \Delta(x, x_0) \le r_n}
\frac{|v(x) - v(x_0)|}{h(r_{n-1})}
\ge a^{-1}C_2^{1/Q} \quad \text{a.s.}
\end{equation}
Note that $h$ is increasing for $r > 0$ small.
For any $r > 0$ small, we can find $n$ large enough such that $r_n \le r \le r_{n-1}$
and $h(r) \le h(r_{n-1})$. Then, by \eqref{Chung:LB-n}, we have
\[
\liminf_{r \to 0+} L(r) \ge a^{-1} C_2^{1/Q} \quad \text{a.s.},
\]
which implies \eqref{Chung:LB} since $a > 1$ is arbitrary.

Now, we turn to the proof of the upper bound \eqref{Chung:UB}. 
It relies on Assumption \ref{a1} which allows us to create independence.
Fix $\delta > 0$ small.
For any $n \ge 1$, let $\rho_n = \exp(-(n^\delta+n^{1+\delta}))$ and 
$b_n = \exp(n^{1+\delta})$.
For any $x \in T$, let $v_n(x) = v([b_n, b_{n+1}), x)$ and 
$\widetilde{v}_n(x) = v(\R_+ \setminus [b_n, b_{n+1}), x)$
so that $v(x) = v_n(x) + \widetilde{v}_n(x)$.
By Assumption \ref{a1}(a), the processes $v_1, v_2, \dots$ are 
independent, and for each $n \ge 1$, $v_n$ and $\widetilde{v}_n$ are also independent.

Let $K := ((1+\delta) C_1)^{1/Q}$.
Since $v(x) = v_n(x) + \widetilde{v}_n(x)$ and $v_n$ and $\widetilde{v}_n$ are 
independent, we can apply Anderson's inequality \cite[Theorem 2]{A55} to get that
\begin{align*}
\P\bigg\{ \sup_{x \in T: \Delta(x, x_0) \le \rho_n} |v_n(x) - v_n(x_0)| \le K h(\rho_n) \bigg\}
\ge \P\bigg\{ \sup_{x \in T: \Delta(x, x_0) \le \rho_n} |v(x) - v(x_0)| \le K h(\rho_n) \bigg\}.
\end{align*}
Then, by Proposition \ref{Prop:SB}, the right-hand side is at least
\begin{align*}
\exp\left(-C_1 K^{-Q} \log\log(1/\rho_n)\right) 
&= (n^\delta + n^{1+\delta})^{-C_1 K^{-Q}}\\
& \ge (2n^{1+\delta})^{-C_1K^{-Q}}.
\end{align*}
Since $(1+\delta)C_1 K^{-Q} = 1$, we have
\begin{align*}
\sum_{n=1}^\infty 
\P\bigg\{ \sup_{x \in T: \Delta(x, x_0) \le \rho_n} |v_n(x) - v_n(x_0)| \le K h(\rho_n) \bigg\}
= \infty.
\end{align*}
Since $v_1, v_2, \dots$ are independent,  the second Borel--Cantelli lemma implies
\begin{equation}\label{Eq:v}
\liminf_{n \to \infty} \sup_{x \in T: \Delta(x, x_0) \le \rho_n}
\frac{|v_n(x) - v_n(x_0)|}{h(\rho_n)} \le 
((1+\delta)C_1)^{1/Q} \quad \text{a.s.}
\end{equation}

To complete the proof of \eqref{Chung:UB}, we claim that
\begin{equation}\label{Eq:tildev}
\limsup_{n \to \infty} \sup_{x \in T: \Delta(x, x_0) \le \rho_n} 
\frac{|\widetilde{v}_n(x) - \widetilde{v}_n(x_0)|}{h(\rho_n)} = 0 \quad \text{a.s.}
\end{equation}
We prove this by using Lemma \ref{Lem:LT}. Consider the process $\widetilde{v}_n$ 
on the set $S_n := B_\Delta(x_0, \rho_n)$.
By \eqref{Eq:a1-1} of Assumption \ref{a1}, for all $x, y \in S_n$,
\begin{equation}\label{Chung:Eq1}
\|\widetilde{v}_n(x) - \widetilde{v}_n(y)\|_{L^2} 
\le c_0 \bigg( \sum_{j=1}^k b_n^{\gamma_j} |x_j - y_j| + b_{n+1}^{-1} \bigg).
\end{equation}
Recall that $\gamma_j = \alpha_j^{-1} - 1$. Let $D_n$ be the diameter of $S_n$ in the metric 
$d_{\widetilde{v}_n}$. Then
\begin{equation}\label{Chung:Eq2}
D_n \le 
C \rho_n \bigg(  \sum_{j=1}^k (b_n \rho_n)^{\alpha_j^{-1} -1}  + (b_{n+1}\rho_n)^{-1} \bigg).
\end{equation}
Note that $b_n \rho_n = \exp(-n^\delta)$.
Also, by the mean value theorem, 
$(n+1)^{1+\delta} - n^{1+\delta} \ge (1+\delta)n^\delta$, which implies
$b_{n+1} \rho_n \ge \exp(\delta n^\delta)$.
Provided $\delta \le \min \{ \alpha_1^{-1} - 1, \dots, \alpha_k^{-1} -1 \}$, 
we have
\begin{equation}\label{Chung:Eq3}
D_n \le C \rho_n \exp(-\delta n^\delta).
\end{equation}
Also, by the independence of $v_n$ and $ \widetilde{v}_n$ and Lemma \ref{Lem:DMX17}, 
for $n$ large, for all $x, y \in S_n$, 
\[
\|\widetilde{v}_n(x) - \widetilde{v}_n(y)\|_{L^2} \le \|v(x) - v(y)\|_{L^2} \le c_1 \Delta(x, y).
\]
This implies $N(S_n, d_{\widetilde{v}_n}, \eps) \le C(\rho_n/\eps)^Q$ for $\eps > 0$ small.
Then for $n$ large,
\begin{align*}
\int_0^{D_n} \sqrt{\log N(S_n, d_{\widetilde{v}_n}, \eps)} \, d\eps
& \le C \int_0^{C \rho_n \exp(-\delta n^\delta)} \sqrt{\log(\rho_n/\eps)}\,d\eps\\
& = C \rho_n \int_0^{C \exp(-\delta n^\delta)} 
\sqrt{\log(1/ \eps)} \, d\eps\\
& \le C \rho_n \exp(-\delta n^\delta) \sqrt{\delta n^\delta}.
\end{align*}
The last inequality can be verified using the change of variable $\eps = e^{-u^2}$ 
and the elementary inequality $\int_x^\infty u^2 e^{-u^2} \, du \le C x e^{-x^2}$
for $x$ large. Let $\zeta > 0$.
Then for $n$ large, we have
\[ 
2 K_0 \zeta h(\rho_n) \ge 
K_0 \Big(\zeta h(\rho_n) + \int_0^{D_n} \sqrt{\log N(S_n, d_{\widetilde{v}_n}, \eps)}\, d\eps\Big),
\]
where $K_0$ is the universal constant in Lemma \ref{Lem:LT}.
Then, by that lemma and \eqref{Chung:Eq3}, we have
\begin{align*}
\P\bigg\{ \sup_{x \in T: \Delta(x, x_0) \le \rho_n} |\widetilde{v}_n(x) - \widetilde{v}_n(x_0)| 
\ge 2K_0 \zeta h(\rho_n) \bigg\}
& \le \exp\bigg(-\frac{\zeta^2 h(\rho_n)^2}{D_n^2}\bigg)\\
& \le \exp\bigg( - \frac{\zeta^2 \exp(2\delta n^\delta)}
{C^2(\log(n^\delta + n^{1+\delta}))^{2/Q}} \bigg).
\end{align*}
Hence
\[ 
\sum_{n=1}^\infty \P\bigg\{ \sup_{x \in T: \Delta(x, x_0) \le \rho_n} 
|\widetilde{v}_n(x) - \widetilde{v}_n(x_0)| \ge 2K_0 \zeta h(\rho_n) \bigg\} < \infty.
\]
By the Borel--Cantelli lemma,
\[ \limsup_{n \to \infty} \sup_{x \in T: \Delta(x, x_0) \le \rho_n} 
\frac{|\widetilde v_n(x) - \widetilde v_n(x_0)|}{h(\rho_n)} \le 2K_0 \zeta \quad 
\text{a.s.} \]
Since $\zeta > 0$ is arbitrary, we get \eqref{Eq:tildev}.

Finally, recall that $v(x) = v_n(x) + \widetilde{v}_n(x)$.
Combining \eqref{Eq:v} and \eqref{Eq:tildev} yields
\[
\liminf_{r \to 0+} L(r) \le ((1+\delta) C_1)^{1/Q} \quad \text{a.s.}
\]
Since $\delta > 0$ is arbitrary, we get \eqref{Chung:UB}.
The proof of Theorem \ref{Thm:ChungLIL} is complete.
\end{proof}

\section{The exact local modulus of continuity}

In this section, we are going to prove the exact local modulus of continuity,
which takes the form of the ordinary LIL.
First, let us recall the following result of Talagrand \cite[Theorem 2.4]{T94}.

\begin{lemma}\label{T94}
Let $\{ X(t), t \in S \}$ be a mean-zero continuous Gaussian process.
Let
\[\sigma^2 := \sup_{t \in S}\|X(t)\|_{L^2}^2.\]
Consider the canonical metric $d_X$ on $S$ defined by 
$d_X(s, t) = \|X(s) - X(t)\|_{L^2}$.
Assume that for some constant $M > \sigma$, some $p > 0$ and 
some $0 < \eps_0 \le \sigma$, we have
\begin{equation*}
N(S, d_X, \eps) \le (M/\eps)^p \quad \text{for all } \eps < \eps_0.
\end{equation*}
Then for $u > \sigma^2[(1+\sqrt{p})/\eps_0]$, we have
\begin{equation*}
\P\bigg\{ \sup_{t \in S} X(t) \ge u \bigg\} 
\le \left( \frac{KMu}{\sqrt{p}\,\sigma^2} \right)^p \Phi\left( \frac{u}{\sigma} \right),
\end{equation*}
where $\Phi(x) = (2\pi)^{-1/2} \int_x^\infty e^{-y^2/2} \, dy$
and $K$ is a universal constant.
\end{lemma}

The following Gaussian estimate is standard:
\begin{equation}\label{stdGaussian}
\frac{1}{2\sqrt{2\pi}x} e^{-x^2/2} \le \Phi(x) \le 
\frac{1}{\sqrt{2\pi}} e^{-x^2/2} \quad \text{for all } x \ge 1.
\end{equation}
Recall that $d(x, y) = \|v(x) - v(y) \|_{L^2}$ is the canonical metric of $v$.
The following theorem gives the exact local modulus of continuity
under the metrics $d$ and $\Delta$, respectively.
Note that the strong LND property (Assumption \ref{a2}) is not required 
for this result.

\begin{theorem}\label{Thm:LIL}
Under Assumptions \ref{a1} and \ref{a3}, for any fixed $x_0 \in T$, we have
\begin{equation}\label{Eq:LIL1}
\lim_{r \to 0+} \sup_{x \in T:\, 0 < d(x, x_0) \le r} \frac{|v(x) - v(x_0)|}{d(x, x_0) 
\sqrt{\log\log(d(x, x_0)^{-1})}} = \sqrt{2}
\quad \text{a.s.}
\end{equation}
and
\begin{equation}\label{Eq:LIL2}
\lim_{r \to 0+} \sup_{x \in T:\, 0 < \Delta(x, x_0) \le r} \frac{|v(x) - v(x_0)|}
{\Delta(x, x_0) \sqrt{\log\log(\Delta(x, x_0)^{-1})}} = \kappa 
\quad \text{a.s.}
\end{equation}
for some positive finite constant $\kappa$ satisfying
\begin{equation*}
\sqrt{2}\, c_3 \le \kappa \le \sqrt{2}\, c_1,
\end{equation*}
where $c_1$ is the constant in \eqref{UB:d} and 
$c_3$ is the constant in Assumption \ref{a3}.
\end{theorem}

\begin{remark}\label{Rmk:LIL}
Meerschaert et al.~\cite{MWX13} have considered Gaussian random fields 
that have stationary increments and satisfy $d(x, y) \asymp \Delta(x, y)$, 
but only proved that
the limit in \eqref{Eq:LIL1} is equal to some finite constant $\kappa_1 \ge \sqrt{2}$.
Our theorem does not require stationarity of increments and we obtain 
the exact constant $\kappa_1 = \sqrt{2}$.
Meerschaert et al.~\cite{MWX13} also proved another form of LIL:
\[ \limsup_{|\eps| \to 0+} \sup_{s:\, |s_j| \le |\eps_j|} \frac{|v(x_0 + s) - v(x_0)|}
{d(s, 0)\sqrt{\log\log(1+\prod_{j=1}^k|s_j|^{-\alpha_j})}} = \kappa_2 \quad \text{a.s.} 
\]
\end{remark}

\begin{proof}[Proof of Theorem \ref{Thm:LIL}]
Fix $x_0 \in T$. For $r > 0$, define 
\[ L(r) := \sup_{x \in T:\, 0 < d(x, x_0) \le r} \frac{|v(x) - v(x_0)|}
{d(x, x_0) \sqrt{\log\log(d(x, x_0)^{-1})}}. \]
By Lemma \ref{0-1law}, $\lim_{r \to 0+} L(r) = \kappa$ a.s.~
for some constant $0 \le \kappa \le \infty$.
To prove \eqref{Eq:LIL1}, we claim that
\begin{equation}\label{LIL:UB}
\limsup_{r \to 0+} L(r) \le \sqrt{2} \quad \text{a.s.}
\end{equation}
and 
\begin{equation}\label{LIL:LB}
\limsup_{r \to 0+} L(r) \ge \sqrt{2} \quad \text{a.s.}
\end{equation}

We first prove the upper bound \eqref{LIL:UB}.
Let $a > 1$ and $\zeta > 0$ be constants. 
For each $n \ge 1$, let 
\[ r_n = a^{-n} \quad \text{and} \quad u_n = (1+\zeta) r_n \sqrt{2 \log\log(1/r_n)}.\]
Consider the event
\[ A_n = \bigg\{ \sup_{x \in T: d(x, x_0) \le r_n} |v(x) - v(x_0)| > u_n \bigg\}. \]
We are going to use Lemma \ref{T94} to derive an upper bound for $\P(A_n)$.
Fix a large $n$.
Consider $S := \{ x \in T : d(x, x_0) \le r_n \}$ and $X(x) := v(x) - v(x_0)$ for $x \in S$.
Then,
\[\sigma^2 := \sup_{x \in S} \|X(x)\|^2_{L^2} = r_n^2\]
and by Lemma \ref{Lem:DMX17}, for all $x, y \in S$,
\[ d_X(x, y) \le c_1 \sum_{j=1}^k |x_j - y_j|^{\alpha_j}. \]
Then $N(S, d_X, \eps) \le C_0(r_n/\eps)^Q$ for $0 < \eps < \sigma$,
where $C_0$ is a constant independent of $\eps$ or $n$,
and can be chosen such that $M := C_0^{1/Q} r_n > \sigma$.
For $n$ large enough, $u_n > r_n(1+\sqrt{Q})$.
Take $\eps_0 = \sigma$ and $p = Q$. 
Then by Lemma \ref{T94}, we have
\begin{align*}
\P(A_n) &\le 2 \bigg( \frac{KC_0^{1/Q}r_n u_n}{\sqrt{Q}\, r_n^2}\bigg)^Q \Phi(u_n/r_n).
\end{align*}
Using the estimate \eqref{stdGaussian},
we get that
\begin{align*}
\P(A_n) & \le C {(\log n)}^{Q/2} \, n^{-(1+\zeta)^2}.
\end{align*}
Hence, $\sum_{n=1}^\infty \P(A_n) < \infty$.
By the Borel--Cantelli lemma,
\begin{align*}
\limsup_{n \to \infty} \sup_{x \in T:\, 0 < d(x, x_0) \le r_n}
\frac{|v(x) - v(x_0)|}{r_n \sqrt{\log\log(1/r_n)}}
\le (1+\zeta)\sqrt{2} \quad \text{a.s.}
\end{align*}
and thus
\[ \limsup_{n \to \infty} \sup_{x \in T:\, r_{n+1}\le d(x, x_0) \le r_n}
\frac{|v(x) - v(x_0)|}{r_{n+1} \sqrt{\log\log(1/r_n)}} \le a (1+\zeta) \sqrt{2} \quad \text{a.s.} 
\]
This implies that
\[ \limsup_{r \to 0+} L(r) \le a (1+\zeta) \sqrt{2} \quad \text{a.s.} \]
Letting $a \downarrow 1$ and $\zeta \downarrow 0$ along rational sequences, 
we get \eqref{LIL:UB}.

We turn to the proof of the lower bound \eqref{LIL:LB}.
Fix $0 < \eps < 1$.
Let $0 < \delta < 1$ be a small fixed number (depending on $\eps$) to be determined.
Write $x_0 = (x_{0, 1}, \dots, x_{0, k})$.
For each $n \ge 1$, let 
\[x_n = \big(x_{0,1} + \rho_n^{\alpha_1^{-1}}, \dots, x_{0,k} + \rho_n^{\alpha_k^{-1}}\big), \]
where $\rho_n = \exp(-(n^\delta + n^{1+\delta}))$.
With Assumption \ref{a1}, we can write $v(x) = v_n(x) + \widetilde v_n(x)$, where
\[ v_n(x) = v([b_n, b_{n+1}), x), \quad
\widetilde{v}_n(x) = v(\R_+ \setminus [b_n, b_{n+1}), x),\] 
and $b_n = \exp(n^{1+\delta})$.
We aim to prove that
\begin{equation}\label{LIL:Eq1}
\limsup_{n \to \infty} \frac{|v_n(x_n) - v_n(x_0)|}
{d(x_n, x_0)\sqrt{\log\log(d(x_n, x_0)^{-1})}} \ge (1-\eps) \sqrt{2} 
\quad \text{a.s.}
\end{equation}
and
\begin{equation}\label{LIL:Eq2}
\limsup_{n \to \infty} \frac{|\widetilde{v}_n(x_n) - \widetilde{v}_n(x_0)|}
{d(x_n, x_0)\sqrt{\log\log(d(x_n, x_0)^{-1})}} \le \eps \quad \text{a.s.}
\end{equation}

To prove \eqref{LIL:Eq1}, we consider for each $n \ge 1$ the event
\[ B_n = \left\{ |v_n(x_n) - v_n(x_0)| \ge 
(1- \eps) d(x_n, x_0) \sqrt{2\log\log(d(x_n, x_0)^{-1})} \right\}. \]
Similarly to \eqref{Chung:Eq1}--\eqref{Chung:Eq3} in the proof of 
Theorem \ref{Thm:ChungLIL}, we can deduce from Assumption \ref{a1} that,
provided $\delta \le \min\{ \alpha_1^{-1}-1, \dots, \alpha_k^{-1}-1 \}$,
\begin{equation}\label{tildev}
\begin{split}
\| \widetilde{v}_n(x_n) - \widetilde{v}_n(x_0) \|_{L^2}
&\le c_0 \Big( \sum_{j=1}^k b_n^{\gamma_j} |x_{n,j} - x_{0,j}| + b_{n+1}^{-1} \Big)\\
&\le K_1 \rho_n \exp(-\delta n^\delta).
\end{split}
\end{equation}
Note that $\Delta(x_n, x_0) = k\rho_n$.
By Assumption \ref{a3} and Lemma \ref{Lem:DMX17},
\begin{equation}\label{d-Delta}
c_3 \Delta(x, x_0) \le d(x, x_0) \le c_1 \Delta(x, x_0)
\end{equation}
for all $x$ in a neighbourhood of $x_0$.
Then, for $n$ large,
\begin{equation}\label{d-rho}
c_3 k \rho_n \le d(x_n, x_0) \le c_1 k \rho_n.
\end{equation}
Therefore, by the triangle inequality, \eqref{tildev} and \eqref{d-rho}, we have
\begin{equation}\label{L2-vn}
\begin{split}
\|v_n(x_n) - v_n(x_0)\|_{L^2} 
& \ge  \|v(x_n) - v(x_0)\|_{L^2} - \| \widetilde{v}_n(x_n) - \widetilde{v}_n(x_0) \|_{L^2}\\
& \ge (1 - K_2 \exp(-\delta n^\delta))d(x_n, x_0).
\end{split}
\end{equation}
Now, \eqref{L2-vn} and \eqref{d-rho} imply that for $n$ large,
\begin{align*}
B_n \supset \left\{ |v_n(x_n) - v_n(x_0)| \ge (1-\eps/2)
\|v_n(x_n) - v_n(x_0)\|_{L^2} \sqrt{2\log\log (C/\rho_n)} \right\},
\end{align*}
where $C$ is a suitable constant.
Then, by the standard Gaussian estimate \eqref{stdGaussian}, 
we get that, for $n$ large,
\begin{align*}
\P(B_n) & \ge K (\log n)^{-1/2} n^{-(1-\eps/2)^2(1+\delta)}.
\end{align*}
By choosing $\delta$ small enough such that $(1-\eps/2)^2(1+\delta)\le 1$,
we have $\sum_{n=1}^\infty \P(B_n) = \infty$.
Hence, by the independence among $v_1, v_2, \dots$ and 
the second Borel--Cantelli lemma, we get \eqref{LIL:Eq1}.
%\[
%\limsup_{n \to \infty} \frac{|v_n(x_n) - v_n(x_0)|}
%{d(x_n, x_0)\sqrt{\log\log(d(x_n, x_0)^{-1})}} \ge (1-\eps)\sqrt{2} 
%\quad \text{a.s.}
%\]

For \eqref{LIL:Eq2}, we use \eqref{tildev} and \eqref{d-rho} 
above to get that
\begin{align*}
& \P\left\{ |\widetilde{v}_n(x_n) - \widetilde{v}_n(x_0)| \ge 
\eps d(x_n, x_0) \sqrt{\log\log(d(x_n, x_0)^{-1})} \right\}\\
& \le \P\left\{ |\widetilde{v}_n(x_n) - \widetilde{v}_n(x_0)| \ge 
K \eps \|\widetilde{v}_n(x_n) - \widetilde{v}_n(x_0)\|_{L^2} \exp(\delta n^\delta)
\sqrt{\log\log(C/\rho_n)} \right\}.
\end{align*}
This probability, by standard Gaussian estimate, is bounded above by
\[ C' \exp\left(-\frac{1}{2}K^2 \eps^2 \exp(2\delta n^\delta) \log\log(C/\rho_n)\right) 
\le C' n^{-2} \]
for $n$ large.
Thus, the Borel--Cantelli lemma implies \eqref{LIL:Eq2}.
Since $v(x) = v_n(x) + \widetilde{v}_n(x)$, combining
\eqref{LIL:Eq1} and \eqref{LIL:Eq2} yields
\[ \limsup_{n \to \infty} \frac{|v(x_n) - v(x_0)|}
{d(x_n, x_0) \sqrt{\log\log(d(x_n, x_0)^{-1})}} \ge (1-\eps)\sqrt{2} - \eps \quad \text{a.s.} 
\]
Since $d(x_n, x_0) \to 0$, this implies 
$\limsup_{r\to0+}L(r) \ge (1-\eps)\sqrt{2} - \eps$ a.s.
Letting $\eps \downarrow 0$ along a rational sequence, we get \eqref{LIL:LB}. 
This completes the proof of \eqref{Eq:LIL1}.
Finally, \eqref{Eq:LIL2} is a direct consequence of 
Lemma \ref{0-1law}, \eqref{Eq:LIL1} and \eqref{d-Delta}.
\end{proof}

\section{The exact uniform modulus of continuity}

The following theorem establishes the exact uniform modulus of continuity for $v$.

\begin{theorem}\label{Thm:MC}
Under Assumptions \ref{a1}, \ref{a2}, and \ref{a3}, we have
\begin{equation}\label{MC-Delta}
\lim_{r \to 0+} \sup_{x, y \in T:\, 0 < \Delta(x, y) \le r}
\frac{|v(x) - v(y)|}{\Delta(x, y) \sqrt{\log(\Delta(x, y)^{-1})}} = \kappa \quad \text{a.s.}
\end{equation}
and
\begin{equation}\label{MC-d}
\lim_{r \to 0+} \sup_{x, y \in T:\, 0 < d(x, y) \le r}
\frac{|v(x) - v(y)|}{d(x, y) \sqrt{\log(d(x, y)^{-1})}} = \kappa' \quad \text{a.s.}
\end{equation}
for some positive finite constants $\kappa$ and $\kappa'$ satisfying
\begin{equation}\label{MC:const}
\sqrt{2 Q c_2} \le \kappa \le \sqrt{2Q}\, c_1 \quad \text{and} \quad
 \sqrt{2 Q c_2} \,c_1^{-1} \le \kappa' \le \sqrt{2Q},
\end{equation}
where $Q = \sum_{j=1}^k \alpha_j^{-1}$, $c_1$ is the constant in \eqref{UB:d} and 
$c_2$ is the constant in Assumption \ref{a2}.
\end{theorem}

\begin{remark}\label{Rmk:MC}
Our Assumption \ref{a2} of strong LND is stronger than condition \emph{(A2)} in 
Theorem 4.1 of Meerschaert et al.~\cite{MWX13}, which is known as the sectorial 
LND property. But our estimates \eqref{MC:const} for the constants are sharper 
than the estimates \emph{(4.3)} in \cite{MWX13}.
\end{remark}

\begin{proof}[Proof of Theorem \ref{Thm:MC}]
For any $r > 0$, let
\[ L(r) := \sup_{x, y \in T:\, 0 < \Delta(x, y) \le r} 
\frac{|v(x) - v(y)|}{\Delta(x, y) \sqrt{\log(\Delta(x, y)^{-1})}}. \]
By Lemma \ref{0-1law}, $\lim_{r \to 0+} L(r) = \kappa$ a.s.~
for some constant $0 \le \kappa \le \infty$.
To prove \eqref{MC-Delta}, we aim to show that
\begin{equation}\label{MC:UB}
\lim_{r \to 0+} L(r) \le \sqrt{2Q} \, c_1 \quad \text{a.s.}
\end{equation}
and
\begin{equation}\label{MC:LB}
\lim_{r \to 0+} L(r) \ge \sqrt{2 Q c_2} \quad \text{a.s.}
\end{equation}

For the upper bound \eqref{MC:UB}, we first prove that 
there is a finite constant $C_0$ such that for any fixed $b > 1$,
\begin{equation}\label{MC:UB1}
\limsup_{n \to \infty} \sup_{x, y \in T:\, \Delta(x, y) \le 2k b^{-n}}
\frac{|v(x) - v(y)|}{b^{-n} \sqrt{\log b^n}} \le C_0 \quad \text{a.s.}
\end{equation}
Indeed, by Theorem 1.3.5 in \cite{AT}, there exists a universal constant $K_0$ such that
for a.e.~$\omega$, there exists $r_0 = r_0(\omega)$ such that for all $0 < r < r_0$,
\[ \sup_{x, y \in T:\, d(x, y) \le r} |v(x) - v(y)| 
\le K_0 \int_0^r \sqrt{\log N(T, d, \eps)} \, d\eps, \]
where $d$ is the canonical metric of $v$.
By Lemma \ref{Lem:DMX17}, $N(T, d, \eps) \le C \eps^{-Q}$ 
for all $\eps > 0$ small, thus for $r > 0$ small,
\[ \int_0^r \sqrt{\log N(T, d, \eps)} \,d\eps \le C r \sqrt{\log(1/r)}. \]
Also, by Lemma \ref{Lem:DMX17}, 
if $\Delta(x, y) \le 2kb^{-n}$ and if $n \ge n_0(\omega)$ is large enough, 
then $d(x, y)$ would be less than $r_0(\omega)$.
Hence, \eqref{MC:UB1} follows immediately.

Of course, \eqref{MC:UB1} implies $\limsup_{r\to0+} L(r) \le \kappa$ for some
finite constant $\kappa$.
In order to improve this and get the sharper bound \eqref{MC:UB},
we use an approximation argument based on anisotropic lattice points.
Let $\eps > 0$ and $1 < a < 2$.
Choose $b$ such that $a < b < a^{1+\eps/(2Q)}$.
Let $n \ge 1$ be an integer. 
For each $i = (i_1, \dots, i_k) \in \Z^k$ and $m = (m_1, \dots, m_k) \in \Z^k$,
define the anisotropic lattice points $z_{n,i}$ and $h_{n,m}$ 
in $\prod_{j=1}^k b^{-n/\alpha_j} \Z$ by
\begin{align*}
z_{n,i} = (i_1 b^{-n/\alpha_1}, \dots, i_k b^{-n/\alpha_k}) \quad
\text{and} \quad  h_{n,m} = (m_1 b^{-n/\alpha_1}, \dots, m_k b^{-n/\alpha_k}).
\end{align*}
Let
\begin{align*}
I_n = \{ i \in \Z^k : z_{n,i} \in T \} \quad \text{and} \quad
M_n = \bigg\{ m \in \Z^k : 
 \sum_{j=1}^k |m_j|^{\alpha_j} b^{-n} \le a^{-n}\bigg\}.
 \end{align*}
Consider the event
\[ A_n = \bigg\{ \max_{i \in I_n} \max_{m \in M_n}
|v(z_{n,i} + h_{n,m}) - v(z_{n,i})| > c_1 a^{-n} \sqrt{2 (Q + \eps) \log a^n}
\bigg\}. \]
By Lemma \ref{Lem:DMX17}, for $i \in I_n$, $m \in M_n$ and $n$ large,
\begin{equation*}
\| v(z_{n,i} + h_{n,m}) - v(z_{n,i}) \|_{L^2} \le c_1 \sum_{j=1}^k |m_j|^{\alpha_j} b^{-n} 
\le c_1 a^{-n}.
\end{equation*}
Also, the cardinality of $I_n$ is $\le C b^{Qn}$ and 
that of $M_n$ is $\le C (b/a)^{Qn}$.
It follows that
\begin{align*}
\P(A_n) 
& \le C b^{Qn} \left( \frac{b}{a} \right)^{Qn} \max_{i \in I_n} \max_{m \in M_n} 
\P\bigg\{ \frac{|v(z_{n,i} + h_{n,m}) - v(z_{n,i})|}{\|v(z_{n,i} + h_{n,m}) - v(z_{n,i})\|_{L^2}} 
> \sqrt{2 (Q + \eps) \log a^n} \bigg\}.
\end{align*}
Then by the standard Gaussian estimate \eqref{stdGaussian}, for $n$ large,
\begin{align*}
\P(A_n) 
& \le C \left( \frac{b^2}{a} \right)^{Qn} \exp(- (Q+\eps)\log a^n) 
= C\left( \frac{b^2}{a^{2+\eps/Q}}\right)^{Qn}.
\end{align*}
The choice of $b$ implies that $\sum_{n=1}^\infty \P(A_n) < \infty$.
By the Borel--Cantelli lemma,
\begin{equation}\label{MC:UB2}
\limsup_{n \to \infty} \max_{i \in I_n} \max_{m \in M_n} 
\frac{|v(z_{n,i} + h_{n,m}) - v(z_{n,i})|}{a^{-n} \sqrt{\log a^n}} 
\le  c_1\sqrt{2(Q+\eps)} \quad \text{a.s.}
\end{equation}

To prove \eqref{MC:UB},
we consider $x, y \in T$ such that $a^{-n-1} \le \Delta(x, y) \le a^{-n}$,
and approximate them by lattice points.
Write $x = (x_1, \dots, x_k)$ and $y = (y_1, \dots, y_k)$.
Choose $i \in I_n$ such that $z_{n,i} = z_{n, i}(x) = (i_1 b^{-n/\alpha_1}, 
\dots, i_k b^{-n/\alpha_k})$ is the lattice point that is closest to $x$.
In particular, for all $j \in\{ 1, \dots, k\}$, we have 
$|x_j - i_j b^{-n/\alpha_j}| \le b^{-n/\alpha_j}$.
Since $\Delta(x, y) \le a^{-n}$, we can also find $m \in M_n$ such that
for all $j \in \{1, \dots, k\}$,
\begin{align*}
m_j b^{-n/\alpha_j} \le y_j - x_j \le (m_j +1) b^{-n/\alpha_j} 
\quad \text{if } y_j - x_j \ge 0,\\
(m_j-1) b^{-n/\alpha_j} \le y_j - x_j \le m_j b^{-n/\alpha_j} 
\quad \text{if } y_j - x_j < 0.
\end{align*}
Let $h_{n,m} = h_{n,m}(x, y) = (m_1 b^{-n/\alpha_1}, \dots, m_k b^{-n/\alpha_k})$ 
and write
\begin{equation*}
v(y) - v(x) = [v(y) - v(z_{n,i} + h_{n,m})] + [v(z_{n,i} + h_{n,m}) - v(z_{n,i})] + 
[v(z_{n,i}) - v(x)].
\end{equation*}
Note that 
\[\Delta(z_{n,i}, x) \le \sum_{j=1}^k |x_j-i_j b^{-n/\alpha_j}|^{\alpha_j} \le k b^{-n}\]
and
\begin{align*}
\Delta(y, z_{n,i}+h_{n,m}) &\le \Delta(y, x+h_{n,m}) + \Delta(x, z_{n,i})\\
& \le \sum_{j=1}^k |y_j - x_j -m_j b^{-n/\alpha_j}|^{\alpha_j} 
+ \sum_{j=1}^k |x_j-i_j b^{-n/\alpha_j}|^{\alpha_j}\\
& \le 2k b^{-n}.
\end{align*}
Then, since $b > a$, \eqref{MC:UB1} implies that
\[ \limsup_{n \to \infty} \sup_{x, y \in T:\, a^{-n-1} \le \Delta(x, y) \le a^{-n}}
\frac{|v(y) - v(z_{n,i} + h_{n,m})|}{a^{-n} \sqrt{\log a^n}} 
= 0 \quad \text{a.s.} \]
and
\[ \limsup_{n \to \infty} \sup_{x, y \in T:\, a^{-n-1} \le \Delta(x, y) \le a^{-n}}
\frac{|v(z_{n,i}) - v(x)|}{a^{-n} \sqrt{\log a^n}} 
= 0 \quad \text{a.s.} \]
Therefore, together with \eqref{MC:UB2}, we have
\[ \limsup_{n \to \infty} \sup_{x, y \in T:\, a^{-n-1} \le \Delta(x, y) \le a^{-n}}
\frac{|v(y) - v(x)|}{a^{-n}\sqrt{\log a^n}} \le c_1\sqrt{2(Q+\eps)} \quad \text{a.s.}
\]
This implies that
\[ \limsup_{r \to 0+} \sup_{x, y \in T:\, 0 < \Delta(x, y) \le r}
\frac{|v(y) - v(x)|}{\Delta(x, y)\sqrt{\log (\Delta(x, y)^{-1})}} 
\le a c_1 \sqrt{2(Q+\eps)} \quad \text{a.s.}
\]
Letting $\varepsilon \downarrow 0$ and $a \downarrow 1$ along rational sequences,
we get the upper bound \eqref{MC:UB}.

Next, we prove the lower bound \eqref{MC:LB} using
the strong LND property from Assumption \ref{a2}. 
Let $T = \prod_{j=1}^k [t_j - s_j, t_j+ s_j]$.
For simplicity, we consider the case where $(t_1, \dots, t_k) \in [0, \infty)^k$
(the proof is similar for other cases).
In this case, it is enough to prove \eqref{MC:LB} with $T$ being replaced by 
$\tilde T = \prod_{j=1}^k [t_j, t_j + s_j]$, whose interior does not contain the origin.
For each $n \ge 1$, let $r_n = 2^{-n}$.
For each $i = (i_1, \dots, i_k) \in \Z^k$, denote 
\[ i-1^* = (i_1 - 1, i_2 \dots, i_k).\]
Define the lattice points $x_{n, i}$ by
\begin{align*}
x_{n, i} = (i_1 2^{-n/\alpha_1}, \dots, i_k 2^{-n/\alpha_k}).
\end{align*}
Let
\begin{align*}
I_n = \{ i \in \Z^k : x_i \in \tilde T \text{ and } x_{i-1^*} \in \tilde T \}
\quad \text{and} \quad
I'_n = \{ i \in \Z^k : x_i \in \tilde T \}.
\end{align*}
Note that $\Delta(x_{n, i}, x_{n, i-1^*}) = r_n$ and the function 
$r \mapsto r\sqrt{\log(1/r)}$ is increasing for $r > 0$ small.
Then
\begin{align*}
\lim_{r \to 0+} L(r) 
& \ge \lim_{n \to \infty} \sup_{x, y \in \tilde T:\, 0 < \Delta(x, y) \le r_n}
\frac{|v(x) - v(y)|}{\Delta(x, y) \sqrt{\log(\Delta(x, y)^{-1})}}\\
&\ge \liminf_{n \to \infty} L_n,
\end{align*}
where
\[ L_n := \max_{i \in I_n} \frac{|v(x_{n,i}) - v(x_{n,i-1^*})|}{r_n \sqrt{\log(1/r_n)}}.
\]
To prove \eqref{MC:LB}, it suffices to prove that 
\begin{equation}\label{MC:LB2}
\liminf_{n \to \infty} L_n \ge \sqrt{2Q c_2} \quad \text{a.s.}
\end{equation}
To this end, let $0 < K < \sqrt{2Q c_2}$.
Fix a large integer $n$ and write $x_i = x_{n,i}$ for simplicity.
We claim that there is a constant $C$ independent of $n$ or $i$ such that 
for all $i \in I_n$, 
\begin{equation}\label{CondP}
\P\bigg\{ \frac{|v(x_{i}) - v(x_{i-1^*})|}{r_n\sqrt{\log(1/r_n)}} \le K\, \bigg|\,
v(x_j) : j \in I'_n \setminus \{ i \}  \bigg\} \le 
\exp\left(-\frac{C2^{-nK^2/(2c_2)}}{\sqrt{n}}\right).
\end{equation}
Indeed, for all $i \in I_n$, by Assumption \ref{a2} and the property that 
the interior of $\tilde T$ does not contain the origin, 
\begin{equation}\label{Eq:cond_var}
\mathrm{Var}\big(v(x_i) \big| v(x_j) : j \in I'_n \setminus \{i\} \big) 
\ge c_2 \min_{j \in I'_n \cup\{0\} \setminus \{i\}} \Delta^2(x_i, x_j)
= c_2\, r_n^2.
\end{equation}
The conditional distribution of $v(x_i)$ given $\{v(x_j) : j \in I'_n \setminus\{i\}\}$ 
is Gaussian with conditional variance 
$\mathrm{Var}(v(x_i) | v(x_j) : j \in I'_n \setminus\{i\})$,
and $v(x_{i-1^*})$ is constant given $\{v(x_j) : j \in I'_n \setminus\{i\}\}$.
Then, by Anderson's inequality \cite{A55} and \eqref{Eq:cond_var}, we have
\begin{align*}
\P\bigg\{\frac{|v(x_i) - v(x_{i-1^*})|}{r_n \sqrt{\log(1/r_n)}} \le K\, \bigg| \,
v(x_j) : j\in I'_n \setminus \{i\} \bigg\}
\le \P\Big\{|Z| \le K \sqrt{c_2^{-1}\log(1/r_n)}\,\Big\},
\end{align*}
where $Z$ is a standard Gaussian random variable. 
Hence, we can derive \eqref{CondP}
using the Gaussian estimate \eqref{stdGaussian} and 
the elementary inequality $1 - x \le \exp(-x)$.

Let $N = N(n)$ be the cardinality of $I_n$.
Order the members of $I_n$ by $i(1), \dots, i(N)$ in a way such that
the value of the first coordinate of $i$ is nondecreasing.
For each $m \in \{1, \dots, N\}$, let $I_n(m) = \{ i(1), \dots, i(m) \}$ and 
consider the event
\[ B_m = \bigg\{ \max_{i \in I_n(m)} \frac{|v(x_i) - v(x_{i-1^*})|}{r_n \sqrt{\log(1/r_n)}} 
\le K \bigg\}. \]
Notice that, for each $2 \le m \le N$, 
the event $B_{m-1}$ depends on the value of process $v$ 
at points among $\{x_{i(1)-1^*}, x_{i(1)}, \dots, x_{i(m-1)-1^*}, x_{i(m-1)}\}$, 
none of which coincides with $x_{i(m)}$
because of the way we order the members of $I_n$.
Therefore, $B_{m-1} \in \sigma\{ v(x_j) : j \in I'_n \setminus \{i(m)\}\}$.
It follows from \eqref{CondP} that
\begin{align*}
\P(B_m)
&= \E\bigg[{\bf 1}_{B_{m-1}}\,
\P\bigg\{\frac{|v(x_{i(m)}) - v(x_{i(m)-1^*})|}{r_n \sqrt{\log(1/r_n)}} \le K\, \bigg| \,
v(x_j) : j \in I'_n \setminus \{i(m)\} \bigg\}\bigg]\\
& \le \P(B_{m-1}) \exp\left(- \frac{C2^{-nK^2/(2c_2)}}{\sqrt{n}}\right).
\end{align*}
Note that $N \sim C 2^{nQ}$. By induction, we get that
\begin{align*}
\P(L_n \le K) = \P(B_N)
\le \exp\left(- C 2^{nQ} \frac{2^{-nK^2/(2c_2)}}{\sqrt{n}}\right).
\end{align*}
Since $Q - K^2/(2c_2) > 0$, we have $\sum_{n=1}^\infty \P(L_n \le K) < \infty$.
Hence, by the Borel--Cantelli lemma, 
\[\liminf_{n \to \infty} L_n \ge K \quad \text{a.s.} \]
Now, we let $K \uparrow \sqrt{2Q c_2}$ along a rational sequence 
to get \eqref{MC:LB2}. This proves \eqref{MC-Delta}.

We turn to the proof of \eqref{MC-d}. Let 
\[ \tilde L(r) := \sup_{x, y \in T:\, 0 < d(x, y) \le r}
\frac{|v(x) - v(y)|}{d(x, y)\sqrt{\log(d(x, y)^{-1})}}. \]
By Lemma \ref{0-1law}, $\lim_{r\to0+}\tilde L(r) = \kappa'$ a.s.~for some constant
$0 \le \kappa' \le \infty$.
Moreover, \eqref{MC-Delta} and Lemma \ref{Lem:DMX17} imply that
\begin{equation*}
\lim_{r\to0+}\tilde L(r) \ge \sqrt{2Q c_2} \, c_1^{-1} \quad \text{a.s.}
\end{equation*}
It remains to prove that
\begin{equation}\label{MC-d:UB}
\lim_{r\to0+}\tilde L(r) \le \sqrt{2Q} \quad \text{a.s.}
\end{equation}
This can be proved by a similar argument that led to \eqref{MC:UB} above.
In fact, the proof of \eqref{MC:UB1} above shows that for any fixed $b > 1$,
\begin{equation}\label{MC-d:UB1}
\limsup_{n \to \infty} \sup_{x, y \in T:\, d(x, y) \le b^{-n}} 
\frac{|v(x) - v(y)|}{b^{-n}\sqrt{\log b^n}} \le C \quad \text{a.s.}
\end{equation}
Let $\eps > 0$, $1 < a < 2$ and $b$ be such that $a < b < a^{1+\eps/(2Q)}$.
We modify the above approximation argument as follows.
For fixed $n$, choose any minimal cover $\{B_d(z_{n,i}, b^{-n}) \}_i$ of $T$
consisting of $d$-balls with centers $z_{n,i} \in T$, and define $I_n = \{z_{n,i}\}_i$.
For each $z_{n, i}$, define $M_{n,i} = \{ h_{n,i,m} \}_m$ such that
$\{ B_d(z_{n,i} + h_{n,i,m}, b^{-n}) \}_m$ is a minimal cover 
of $B_d(z_{n,i}, (1+\eps)a^{-n})$.
Consider the event
\[ A_n = \Big\{\max_{z_{n,i} \in I_n} \max_{h_{n,i,m} \in M_{n,i}} |v(z_{n,i} + h_{n,i,m}) 
- v(z_{n,i})| > (1+\eps) a^{-n} \sqrt{2(Q+\eps) \log a^n}\Big\}. 
\]
Since $d$ is comparable to $\Delta$ by Assumption \ref{a3} and Lemma
 \ref{Lem:DMX17}, the cardinality of $I_n$ is $\le Cb^{nQ}$ and 
that of $M_{n,i}$ is $\le C(b/a)^{nQ}$.
Also, $\|v(z_{n,i} + h_{n,i,m}) - v(z_{n,i})\|_{L^2} \le (1+\eps)a^{-n}$.
Then, as before, we can show that $\sum_{n=1}^\infty \P(A_n) < \infty$ and
\begin{equation}\label{MC-d:UB2}
\limsup_{n\to\infty} \max_{z_{n,i} \in I_n} \max_{h_{n,i,m} \in M_{n,i}} 
\frac{|v(z_{n,i} + h_{n,i,m}) - v(z_{n,i})|}{a^{-n}\sqrt{\log a^n}} 
\le (1+\eps) \sqrt{2(Q+\eps)} 
\quad \text{a.s.}
\end{equation}
Consider $x, y \in T$ such that $a^{-n-1} \le d(x, y) \le a^{-n}$.
Then, we can find $z_{n,i} = z_{n,i}(x) \in I_n$ such that $d(z_{n,i}, x) \le b^{-n}$.
Since $d(z_{n,i}, y) \le d(z_{n,i}, x) + d(x, y) \le b^{-n} + a^{-n} \le (1+\eps) a^{-n}$
for $n$ large, we can also find $h_{n,i,m} = h_{n,i,m}(x, y) \in M_{n,i}$ such that
$d(z_{n,i} + h_{n,i,m}, y) \le b^{-n}$. Then, by \eqref{MC-d:UB1} and $b > a$, 
we get
\[ \limsup_{n\to \infty} \sup_{x, y \in T:\, a^{-n-1} \le d(x, y) \le a^{-n}}
\frac{|v(z_{n,i}) - v(x)|}{a^{-n}\sqrt{\log a^n}} = 0 \quad \text{a.s.} \]
and
\[ \limsup_{n\to \infty} \sup_{x, y \in T:\, a^{-n-1} \le d(x, y) \le a^{-n}}
\frac{|v(z_{n,i} + h_{n,i,m}) - v(y)|}{a^{-n}\sqrt{\log a^n}} = 0 \quad \text{a.s.} \]
Combining this with \eqref{MC-d:UB2}, we get
\[ 
\limsup_{n \to \infty} \sup_{x, y \in T:\, a^{-n-1}\le d(x, y) \le a^{-n}} 
\frac{|v(y) - v(x)|}{a^{-n} \sqrt{\log a^n}} \le (1+\eps) \sqrt{2(Q+\eps)} 
\quad \text{a.s.} \]
which implies that
\[ \limsup_{r \to 0+} \sup_{x, y \in T:\, 0 < d(x, y) \le r} 
\frac{|v(y) - v(x)|}{d(x, y) \sqrt{\log(d(x, y)^{-1})}} \le (1+\eps)a \sqrt{2(Q+\eps)} 
\quad \text{a.s.} \]
Letting $\eps \downarrow 0$ and $a \downarrow 1$ yields \eqref{MC-d:UB}.
This completes the proof of Theorem \ref{Thm:MC}.
\end{proof}

\section{Linear SPDEs driven by fractional-colored noise}

In this section, we give an application of our main results to a class of linear SPDEs. 
Consider the equation
\begin{equation}\label{Eq:SHE}
\frac{\partial}{\partial t} u(t, x) = \mathscr{L}u(t,x) + \dot{W}(t, x), \quad t \ge 0,\ x \in \mathbb{R}^d,
\end{equation}
with zero initial condition $u(0, x) = 0$. Here, $\mathscr{L}$ is the infinitesimal 
generator of a symmetric L\'{e}vy process $X = \{ X(t), t \ge 0 \}$ taking values 
in $\R^d$, and $\dot{W}$ is a fractional-colored (or white-colored) centered Gaussian 
noise with Hurst index $1/2 \le H < 1$ in time and spatial covariance $f$, i.e.,
\[ \E[\dot{W}(t, x)\dot{W}(s,y)] = \rho_H(t-s) f(x-y), \]
where
\[ \rho_H(t-s) = \begin{cases}
a_H |t-s|^{2H-2} & \text{if }1/2 < H < 1,\\
\delta(t-s) & \text{if } H = 1/2,
\end{cases}\]
and where $a_H = H(2H-1)$ and $\delta$ is the delta function. When $X$ is a Brownian motion, 
$\mathscr{L}$ is the Laplace operator and \eqref{Eq:SHE} is the stochastic heat equation.
Furthermore, when $H = 1/2$, %i.e., $\dot{W}$ is white in time,
\eqref{Eq:SHE} is the stochastic heat equation considered in \cite{DMX17}.

The existence of the solution to \eqref{Eq:SHE} has been studied in \cite{BT08, HSWX} 
(and in \cite{D99} for $H = 1/2$), and the space-time regularity of the solution has been 
studied in \cite{TX17, HSWX}. Herrell et al.~\cite{HSWX} used the the idea of string 
processes of Mueller and Tribe \cite{MT02} and showed that $\{u(t, x), t \ge 0, x \in \R^k\}$ 
admits the decomposition
\begin{equation}\label{SHE:de}
u(t, x) = U(t, x) - Y(t, x),
\end{equation}
where $\{U(t, x),  t \ge 0, x \in \R^k\}$ has stationary increments and satisfies 
the property of 
strong LND, %with respect to $\Delta$,
while $\{Y(t, x),  t \ge 0, x \in \R^k\}$ has smooth sample
paths. Consequently, certain regularity properties of $u(t, x)$ can be deduced 
from those of $U(t, x)$. Now we can deal with $\{u(t, x), t \ge 0, x \in \R^k\}$ directly.

We assume that $f$ is the Fourier transform of a tempered measure $\mu$ 
which is absolutely continuous with respect to the Lebesgue measure
with density $h$, i.e., $\mu(d\xi) = h(\xi)d\xi$.
A typical example is $h(\xi) = |\xi|^{-\beta}$, $0 < \beta < d$. 
In this case, $f$ is called the Riesz kernel: $f(x) = C|x|^{\beta-d}$,
where $C$ is some suitable constant depending on $\beta$ and $d$; 
see \cite[\S V]{S}.

Let $\Psi(\xi)$ be the characteristic exponent of $X$ given by
\[ \E [e^{i  \xi \cdot X(t)}] = e^{-t\Psi(\xi)}, \quad t \ge 0,\ \xi \in \R^d. \]
Note that $\Psi(\xi) = \Psi(-\xi) \ge 0$ for all $\xi \in \R^d$ 
since $X$ is assumed to be symmetric.
Assume that $X(t)$ has a probability density function given by
\begin{equation}\label{X:pdf}
p_t(x) = \frac{1}{(2\pi)^d} \int_{\R^d} e^{-i \xi \cdot x}e^{-t\Psi(\xi)} d\xi, 
\quad t > 0,\ x \in \R^d.
\end{equation}

Let $T_0 > 0$. Recall that the Gaussian noise $W$ defines a linear isometry 
from the Hilbert space completion $\mathcal{HP}$ of the space $C^\infty_c
((0, T_0) \times \R^d)$ of compactly supported smooth functions with respect 
to the inner product $\langle \cdot, \cdot \rangle_{\mathcal{HP}}$ into the Gaussian 
space in $L^2(\P)$:
\begin{align}
\begin{aligned}\label{W:iso}
\varphi & \mapsto  W(\varphi) := \int_0^{T_0}\int_{\R^d} \varphi(s, y) W(ds, dy),\\
& \qquad \E[W(\varphi)W(\psi)] = \langle \varphi, \psi\rangle_{\mathcal{HP}}.
\end{aligned}
\end{align}
For test functions $\varphi, \psi$ on $(0, T_0) \times \R^d$, the inner product 
${\langle \varphi, \psi\rangle}_{\mathcal{HP}}$ is defined by
\begin{align}
\begin{aligned}\label{HP}
&{\langle \varphi, \psi\rangle}_{\mathcal{HP}}
 := \int_0^{T_0} \int_0^{T_0} ds\, dr \int_{\R^d} \int_{\R^d} dy\, dz\, \varphi(s, y) 
%|s-r|^{2H-2}
\rho_H(s-r) f(y-z) \psi(r, z)\\
& = b_H \int_{\R} d\tau \, |\tau|^{1-2H} \int_{\R^d} \int_{\R^d} dy\, dz\, f(y-z) 
\mathcal{F}(\varphi(\cdot, y){\bf 1}_{[0, T_0]}(\cdot))(\tau) 
\overline{\mathcal{F}(\psi(\cdot, z){\bf 1}_{[0, T_0]}(\cdot))(\tau)}\\ 
& = c_{H,d} \int_{\R} d\tau \, |\tau|^{1-2H} \int_{\R^d} d\xi \, h(\xi)
\mathcal{F}(\varphi{\bf 1}_{[0, T_0]})(\tau, \xi) \overline{\mathcal{F}(\psi{\bf 1}_{[0, T_0]})(\tau, \xi)},
\end{aligned}
\end{align}
where $b_H = a_H(2^{2(1-H)}\sqrt{\pi}\,)^{-1}\Gamma(H-1/2)/\Gamma(1-H)$
and $c_{H,d} = b_H (2\pi)^{-d}$, for $1/2 < H < 1$; see \cite{BT08, HSWX}.
In fact, the equalities in \eqref{HP} also hold for $H = 1/2$ with 
$b_{1/2} = (2\pi)^{-1}$.
In the above, $\mathcal{F}$ denotes the Fourier transform defined, 
for any functions $g: \R \to \C$ and $\varphi: \R^{1+d} \to \C$, by
\[ \mathcal{F}g(\tau) = \int_\R e^{-i\tau s} g(s)\, ds, \quad 
\mathcal{F}\varphi(\tau, \xi) = \int_{\R^{1+d}} e^{-i\tau s - i\xi \cdot y} \varphi(s, y) \,ds\,dy. \]
If follows from \cite{HSWX} and \cite{D99} 
(for $1/2 < H < 1$ and $H = 1/2$ respectively) that if
\begin{equation}\label{Dalang_cond}
\int_{\R^d} \frac{\mu(d\xi)}{1 + \Psi(\xi)^{2H}} < \infty,
\end{equation}
then \eqref{Eq:SHE} has a random field solution on $[0, T_0]$ which is given by
\[ u(t, x) = \int_0^t \int_{\R^d} p_{t-s}(x-y) W(ds, dy). \]

Denote $G_{t, x}(s, y) = p_{t-s}(x-y) {\bf 1}_{[0, t]}(s)$. It follows from \eqref{HP} that 
for any $a_1, \dots, a_n \in \R$, for any $t^1, \dots, t^n \in [0, T_0]$ and
$x^1, \dots, x^n \in \R^d$, we have
\begin{align}\label{Eq:var_lin_comb}
\E\Bigg[ \bigg(\sum_{j=1}^n a_j u(t^j, x^j) \bigg)^2 \Bigg] 
= c_{H, d} \int_{\R} d\tau \, |\tau|^{1-2H} \int_{\R^d} d\xi\, h(\xi) \, 
|\mathcal{F}G(\tau, \xi)|^2,
\end{align}
where $G = \sum_{j=1}^n a_j G_{t^j, x^j}$. 
Note that $p_t(\cdot)$ is equal to the inverse Fourier transform of 
$\xi \mapsto e^{-t\Psi(\xi)}$ since $\Psi(\xi) = \Psi(-\xi)$.
Hence, it can be verified that the Fourier transform of $G_{t, x}(\cdot, \cdot)$ is
\begin{equation}\label{Eq:FT_g}
\mathcal{F}G_{t, x}(\tau, \xi) =  \frac{e^{-i \xi \cdot x}(e^{-i\tau t} - e^{-t \Psi(\xi)})}
{\Psi(\xi) - i\tau}, \quad \tau \in \R,\ \xi \in \R^d.
\end{equation}

Dalang et al.~\cite{DMX17} have established a harmonizable representation for the 
solution of the stochastic heat equation
\begin{equation*}
\frac{\partial}{\partial t} u(t, x) = \Delta u(t, x) + \dot W(t, x),
\end{equation*}
where $\dot W$ is a spatially homogeneous Gaussian noise that is white in time 
and colored in space.  In the following, we follow the approach of \cite{DMX17} to 
establish a similar representation for the solution of equation \eqref{Eq:SHE} driven 
by the fractional-colored Gaussian noise.

Let $\tilde W_1, \tilde W_2$ be independent space-time Gaussian white 
noise on $\R\times \R^d$. Let $\tilde W = \tilde W_1 + i \tilde W_2$.
For each $(t, x) \in [0, T_0] \times \R^d$, define
\begin{equation}\label{Def:v}
v(t, x) = c_{H,d}^{1/2}\, \Re
\iint_{\R \times \R^d} \mathcal{F}G_{t, x}(\tau, \xi)\, |\tau|^{\frac{1-2H}{2}} 
h^{\frac 1 2}(\xi) \tilde W(d\tau, d\xi).
\end{equation}
The following lemma verifies that $v(t, x)$ has the same law as the solution $u(t, x)$ 
of equation \eqref{Eq:SHE}. We will call \eqref{Def:v} the harmonizable representation 
of $u(t, x)$.

\begin{lemma}\label{SHE:law}
The Gaussian random field $v = \{v(t, x), (t, x) \in [0, T_0] \times \R^d \}$ 
has the same law as the solution $u = \{ u(t, x), (t, x) \in [0, T_0] \times \R^d \}$ 
of equation \eqref{Eq:SHE}.
\end{lemma}

\begin{proof}
It is clear that $v$ is Gaussian.
By \eqref{HP}, for any $(t, x), (s, y) \in [0, T_0] \times \R^d$,
\begin{align*}
\E[v(t, x) v(s, y)] 
& = \iint_{\R \times \R^d} \mathcal{F}G_{t, x}(\tau, \xi)\, 
\overline{\mathcal{F}G_{s, y}(\tau, \xi)}\, |\tau|^{1-2H} h(\xi) \, d\tau \, d\xi\\
& = \langle G_{t, x}, G_{s, y} \rangle_{\mathcal{HP}}\\
& = \E[u(t, x) u(s, y)].
\end{align*}
Hence $v$ and $u$ have the same law.
\end{proof}

From now on, suppose that there exist positive finite constants $c_\Psi$ and $C_\Psi$ 
such that
\begin{equation}\label{Psi:cond}
c_\Psi|\xi|^\alpha \le \Psi(\xi) \le C_\Psi |\xi|^\alpha 
\text{ for all } \xi \in \R^d, \quad \text{where } 0 < \alpha \le 2,
\end{equation}
and there exist positive finite constants $c_h$ and $C_h$ such that
\begin{equation}\label{h:cond}
c_h |\xi|^{-\beta} \le h(\xi) \le C_h |\xi|^{-\beta}
\text{ for all } \xi \in \R^d, \quad \text{where } 0 < \beta < d.
\end{equation}
Define
\begin{equation*}\label{theta}
\theta_1 := H - \frac{d-\beta}{2\alpha} \quad \text{and} \quad 
\theta_2 := \alpha \theta_1 = \alpha H - \frac{d-\beta}{2}.
\end{equation*}
These are the H\"older exponents of $u(t, x)$ in time and space respectively.
By \eqref{Dalang_cond}, if $\beta > d-2\alpha H$, or equivalently, $\theta_1 > 0$,
then \eqref{Eq:SHE} has a solution.
Consider the following metric on $\R \times \R^d$:
\begin{equation}\label{SHE:Delta}
\Delta((t, x), (s, y)) = |t-s|^{\theta_1} + \sum_{j=1}^d |x_j - y_j|^{\theta_2}.
\end{equation}
We always have $\theta_1 < 1$ since $H < 1$ and $\beta < d$.
Furthermore, we assume the following condition:
\begin{equation}\label{theta:cond}
\theta_1 > 0 \quad \text{and} \quad \theta_2 < 1.
\end{equation}
The condition $\theta_2 < 1$ is equivalent to $\beta < d-2\alpha H + 2$.
%Under this condition, the values of the parameters $H, \alpha, \beta$ satisfy:
%$H \in (\frac{d-\beta}{2\alpha}, \frac{d-\beta+2}{2\alpha}) \cap [\frac 1 2, 1)$,
%$\alpha \in (\frac{d-\beta}{2H}, \frac{d-\beta+2}{2H}) \cap (0, 2]$ and 
%$\beta \in (d - 2\alpha H, d - 2\alpha H + 2) \cap (0, d)$.

We now verify Assumptions \ref{a1} and \ref{a2} for the solution of \eqref{Eq:SHE}
using the harmonizable representation \eqref{Def:v}.

\begin{lemma}\label{Lem:SHE-a1}
Suppose $\Psi$ and $h$ satisfy conditions \eqref{Psi:cond} and \eqref{h:cond} 
respectively. Suppose $\theta_1$ and $\theta_2$ satisfy condition \eqref{theta:cond}.
Let $T$ be a compact rectangle in $(0, \infty) \times \R^d$.
Then the Gaussian random field $\{v(A, t, x), A \in \mathscr{B}(\R_+), (t, x) \in T \}$ 
defined by
\begin{equation}\label{Eq:SHE-HR}
v(A, t, x) = c_{H,d}^{1/2} \, \Re \iint_{\{(\tau, \xi) : \max(|\tau|^{\theta_1}, |\xi|^{\theta_2}) \in A\}}
\mathcal{F}G_{t, x}(\tau, \xi)\, |\tau|^{\frac{1-2H}{2}} h^{\frac 1 2} (\xi)
\tilde W(d\tau, d\xi)
\end{equation}
satisfies Assumption \ref{a1}(a). 
Moreover, there exists a finite constant $c_0$ such that for all $0 \le a < b \le \infty$
and all $(t_0, x_0), (t, x) \in T$,
\begin{align}
\begin{aligned}\label{Eq:SHE-a1}
&{\|v([a, b), t, x) - v(t, x) - v([a, b), t_0, x_0) + v(t_0, x_0)\|}_{L^2}\\
&\le c_0 \bigg( a^{\gamma_1} |t - t_0| + a^{\gamma_2}\sum_{j=1}^d |x_j - x_{0, j}| + b^{-1}\bigg),
\end{aligned}
\end{align}
where $\gamma_1 = \theta_1^{-1} - 1$ and $\gamma_2 = \theta_2^{-1} - 1$.
In particular, Assumption \ref{a1}(b) is satisfied for $a_0 = 0$.
\end{lemma}

\begin{proof}
It is obvious that $v(A, t, x)$ satisfies part (a) of Assumption \ref{a1}.
For part (b), the proof is similar to that of Lemma 7.3 in \cite{DMX17}:
First,
\begin{align*}
&v([a, b), t, x) - v(t, x) - v([a, b), t_0, x_0) + v(t_0, x_0)\\
&= v([0, a), t_0, x_0) - v([0, a), t, x)
+ v([b, \infty), t_0, x_0) - v([b, \infty), t, x).
\end{align*}
By \eqref{Eq:var_lin_comb} and \eqref{Eq:FT_g},
\begin{align}
\begin{split}\label{Eq:a}
&\E[(v([0, a), t, x) - v([0, a), t_0, x_0))^2]\\
%& = c_{H,d} \iint_{D_1(a)} |\mathcal{F}G_{t, x}(\tau, \xi) - 
%\mathcal{F}G_{t_0, x_0}(\tau, \xi)|^2 |\tau|^{1-2H} h(\xi)\, d\tau\, d\xi\\
& = C \iint_{D_1(a)} \bigg|\frac{(e^{-i\tau t} - e^{-t\Psi(\xi)}) 
- e^{-i\xi\cdot(x_0-x)}(e^{-i\tau t_0} - e^{-t_0\Psi(\xi)})}{\Psi(\xi) - i\tau}\bigg|^2 
|\tau|^{1-2H} h(\xi)\, d\tau\, d\xi\\
& = C \iint_{D_1(a)} \frac{\varphi_1(t, x, \tau, \xi)^2 + 
\varphi_2(t, x, \tau, \xi)^2}{\Psi(\xi)^2 + |\tau|^2} |\tau|^{1-2H} h(\xi)\, d\tau\, d\xi
\end{split}
\end{align}
and 
\begin{align}
\begin{split}\label{Eq:b}
&\E[(v([b, \infty), t, x) - v([b, \infty), t_0, x_0))^2]\\
& = C \iint_{D_2(b)} \frac{\varphi_1(t, x, \tau, \xi)^2 + \varphi_2(t, x, \tau, \xi)^2}
{\Psi(\xi)^2 + |\tau|^2} |\tau|^{1-2H} h(\xi)\, d\tau\, d\xi,
\end{split}
\end{align}
where
\begin{align*}
& D_1(a) 
= \{ (\tau, \xi) \in \R \times \R^d : \max(|\tau|^{\theta_1}, |\xi|^{\theta_2}) < a \},\\
& D_2(b) 
= \{ (\tau, \xi) \in \R \times \R^d : \max(|\tau|^{\theta_1}, |\xi|^{\theta_2}) \ge b \},\\
&\varphi_1(t, x, \tau, \xi) 
= \cos(\tau t) - e^{-t\Psi(\xi)} - \cos(\xi \cdot(x_0 - x) + \tau t_0) 
+ e^{-t_0\Psi(\xi)}\cos(\xi \cdot (x_0 - x)),\\
&\varphi_2(t, x, \tau, \xi)
= -\sin(\tau t) + \sin(\xi\cdot (x_0 - x) + \tau t_0) - e^{-t_0\Psi(\xi)} \sin(\xi \cdot (x_0 - x)).
\end{align*}
Consider \eqref{Eq:a}.
Note that $\varphi_1(t_0, x_0, \tau, \xi) = 0 = \varphi_2(t_0, x_0, \tau, \xi)$, and
\begin{align*}
|\partial_t \varphi_j| \le |\tau| + \Psi(\xi) \quad \text{and} \quad
|\partial_x \varphi_j| \le 2|\xi|, \quad j = 1, 2.
\end{align*}
Then, by the mean value theorem,
\begin{align*}
&\E[(v([0, a), t, x) - v([0, a), t_0, x_0))^2]\\
& \le C \iint_{D_1(a)} \Big( 4(|\tau|^2 + \Psi(\xi)^2)|t - t_0|^2 + 8|\xi|^2 |x - x_0|^2 \Big) 
\frac{|\tau|^{1-2H} |\xi|^{-\beta}}{\Psi(\xi)^2 + |\tau|^2}  d\tau\, d\xi\\
& = 4 C  |t- t_0|^2 \iint_{D_1(a)} |\tau|^{1-2H} |\xi|^{-\beta} d\tau\, d\xi
+ 8 C |x - x_0|^2 \iint_{D_1(a)}\frac{|\tau|^{1-2H} |\xi|^{2-\beta}}{\Psi(\xi)^2 + |\tau|^2}
d\tau\, d\xi\\
& =: 4C|t- t_0|^2 I_1 + 8C|x-x_0|^2 I_2.
\end{align*}
Using polar coordinates $r = |\xi|$, 
\begin{align*}
I_1 &= C
\iint_{\{(\tau, r) \in \R\times \R_+: \max(|\tau|^{\theta_1}, r^{\theta_2}) < a\}} 
|\tau|^{1-2H} r^{d-\beta-1} d\tau\,dr\\
& \le 
C \int_{-a^{\theta_1^{-1}}}^{a^{\theta_1^{-1}}} d\tau \, |\tau|^{1-2H} 
\int_0^{a^{\theta_2^{-1}}} dr \, r^{d-\beta-1}\\
& = C a^{(2-2H)\theta_1^{-1}+(d-\beta)\theta_2^{-1}}.
\end{align*}
Since $\theta_2 = \alpha \theta_1$ and $\gamma_1 = \theta_1^{-1}-1$, 
we get that $I_1 \le C a^{2\gamma_1}$.

For $I_2$, we use the condition $c|\xi|^\alpha \le \Psi(\xi) \le C|\xi|^\alpha$,
polar coordinates $r = |\xi|$, and symmetry of the integrand in $\tau$ to get that
\begin{align*}
I_2 &\le 
C \iint_{\{(\tau, r) \in \R_+^2 :\max(\tau^{\theta_1}, r^{\theta_2}) < a\}} 
\frac{\tau^{1-2H}r^{d-\beta-1}}{(c^2 \wedge 1)(r^{2\alpha} + \tau^2)} d\tau\,dr.
\end{align*}
Putting $\rho = r^\alpha$ and $|z|$ the Euclidean norm of 
$z = (\tau, \rho)$ in $\R^2$, we get that
\begin{align*}
I_2 & \le C \iint_{\{(\tau, \rho) \in \R_+^2 : \max(\tau, \rho) < a^{\theta_1^{-1}}\}} 
\frac{\tau^{1-2H}\rho^{\alpha^{-1}(d-\beta+2) - 1}}{\rho^2 + \tau^2} d\tau\, d\rho\\
& \le C \iint_{\{z \in \R_+^2 : |z| < \sqrt 2 a^{\theta_1^{-1}}\}} 
|z|^{-2H+\alpha^{-1}(d-\beta+2)-2} dz\\
& = C a^{\theta_1^{-1}(-2H + \alpha^{-1} (d-\beta+2))} = C a^{2\gamma_2}.
\end{align*}
Therefore, we have
\begin{equation}\label{BD:a}
\E[(v([0, a), t, x) - v([0, a), t_0, x_0))^2]
\le C \left( a^{2\gamma_1}|t-t_0|^2 + a^{2\gamma_2}|x-x_0|^2 \right).
\end{equation}

For \eqref{Eq:b}, we can use the bounds $|\varphi_1| \le 4$ and $|\varphi_2| \le 3$ 
to deduce that
\begin{align*}
\E[(v([b, \infty), t, x) - v([b, \infty), t_0, x_0))^2]
\le C \iint_{D_2(b)} \frac{|\tau|^{1-2H} |\xi|^{-\beta}}{c^2|\xi|^{2\alpha} + |\tau|^2}
d\tau\, d\xi.
\end{align*}
To estimate the above integral, we split $D_2(b)$ into two parts:
\begin{equation*}
D_2(b) = \big\{ |\tau|^{\theta_1} \le |\xi|^{\theta_2}, |\xi|^{\theta_2} \ge b \big\} \cup 
\big\{ |\tau|^{\theta_1} > |\xi|^{\theta_2}, |\tau|^{\theta_1} \ge b \big\}.
\end{equation*}
Passing to polar coordinates, we have
\begin{align*}
&\iint_{D_2(b)} \frac{|\tau|^{1-2H} |\xi|^{-\beta}}{c^2|\xi|^{2\alpha} + |\tau|^2}
d\tau\, d\xi\\
& \le C \iint_{\{ |\tau|^{\theta_1} \le r^{\theta_2}, r^{\theta_2} \ge b \}} 
\frac{|\tau|^{1-2H}r^{d-\beta-1}}{c^2r^{2\alpha}} d\tau\, dr
+ C \iint_{\{ |\tau|^{\theta_1} > r^{\theta_2}, |\tau|^{\theta_1} \ge b \}}
\frac{|\tau|^{1-2H}r^{d-\beta-1}}{|\tau|^2} d\tau\, dr\\
& = C\int_{b^{\theta_2^{-1}}}^\infty dr\, r^{d-\beta-1-2\alpha} 
\int_{-r^\alpha}^{r^\alpha} d\tau\, |\tau|^{1-2H}
+ C \int_{b^{\theta_1^{-1}}}^\infty d\tau \, |\tau|^{-1-2H} 
\int_0^{|\tau|^{\alpha^{-1}}} dr\, r^{d-\beta-1}\\
& = C b^{-2}.
\end{align*}
We have shown that
\begin{equation}\label{BD:b}
\E[(v([b, \infty), t, x) - v([b, \infty), t_0, x_0))^2] \le C b^{-2}.
\end{equation}
Therefore, \eqref{Eq:SHE-a1} follows immediately from \eqref{BD:a} and \eqref{BD:b}.
\end{proof}

%Next, we verify Assumption \ref{a2}, namely, the strong LND property for the solution
%of equation \eqref{Eq:SHE} with respect to the metric $\Delta$ defined in \eqref{SHE:Delta}. 
The following result shows that $u(t, x)$ satisfies strong LND with respect to 
the metric $\Delta$ defined in \eqref{SHE:Delta} above, thus 
verifies Assumption \ref{a2}.
%strengthening the result of \cite{HSWX}.

\begin{lemma}\label{Lem:SHE-a2}
Suppose $\Psi$ and $h$ satisfy conditions \eqref{Psi:cond} and \eqref{h:cond} 
respectively. Suppose $\theta_1$ and $\theta_2$ satisfy \eqref{theta:cond}. 
Let $T$ be a compact rectangle in $(0, \infty) \times \R^d$.
Then there exists a constant $c_2 > 0$ such that for any $n \ge 1$, for any
$(t, x), (t^1, x^1), \dots, (t^n, x^n) \in T$, we have
\begin{equation*}
\mathrm{Var}(u(t, x)|u(t^1, x^1), \dots, u(t^n, x^n)) 
\ge c_2 \min_{1 \le j \le n}\Delta^2((t,x),(t^j, x^j)).
\end{equation*}
In particular, this implies that 
$\|u(t, x) - u(s, y)\|_{L^2} \ge \sqrt{c_2} \Delta((t, x), (s, y))$ for all $(t, x), (s, y) \in T$.
\end{lemma}

\begin{proof}
We may assume that $T = [a, a'] \times [-b, b]^d$, 
where $0 < a < a' < \infty$ and $0 < b < \infty$.
It suffices to show that there exists a positive constant $C$ such that
\[\E\Bigg[ \bigg( u(t, x) - \sum_{j=1}^n a_j u(t^j, x^j) \bigg)^2 \Bigg] \ge Cr^{2\theta_2},\]
for any $n \ge 1$, any $(t, x), (t^1, x^1), \dots, (t^n, x^n) \in T$, 
and any $a_1, \dots, a_n \in \R$, where
\[ r = \min_{1 \le j \le n}(|t - t^j|^{1/\alpha} \vee |x - x^j|). \]
From \eqref{Eq:var_lin_comb} and \eqref{Eq:FT_g}, we see that
\begin{align}\label{Eq:var_integral}
&\E\Bigg[ \bigg( u(t, x) - \sum_{j=1}^n a_j u(t^j, x^j) \bigg)^2 \Bigg]\\
& \ge K_0 \int_{\R} d\tau \int_{\R^d} d\xi \, 
\bigg|e^{-i \xi \cdot x}(e^{-i\tau t} - e^{-t \Psi(\xi)}) 
- \sum_{j=1}^n a_j e^{-i \xi \cdot x^j}(e^{-i\tau t^j} - e^{-t^j \Psi(\xi)})\bigg|^2 
\frac{|\tau|^{1-2H}|\xi|^{-\beta}}{C_\Psi^2|\xi|^{2\alpha} + |\tau|^2},
\notag
\end{align}
where $K_0 = c_{H,d}\, c_h$.
Let $M$ be a finite constant such that 
$|t - t'|^{1/\alpha} \vee |x-x'| \le M$ for all $(t, x), (t',x') \in T$.
Let $\rho = \min\{a/M^\alpha, 1\}$.
Choose and fix two nonnegative smooth test functions
$f: \R \to \R_+$ and $g: \R^d \to \R_+$ 
which vanish outside $[-\rho, \rho]$ and the unit ball respectively, and 
satisfy $f(0) = g(0) = 1$.
Let $f_r(\tau) = r^{-\alpha} f(r^{-\alpha}\tau)$, $g_r(\xi) = r^{-d} g(r^{-1} \xi)$,
and denote the Fourier transforms of $f_r$ and $g_r$ by 
$\widehat{f_r}$ and $\widehat{g_r}$ respectively.
Consider the integral
\begin{align*}
I := 
\int_{\R} d\tau \int_{\R^d} d\xi \bigg[e^{-i \xi \cdot x}(e^{-i\tau t} - e^{-t \Psi(\xi)}) 
- \sum_{j=1}^n a_j e^{-i \xi \cdot x^j}(e^{-i\tau t^j} - e^{-t^j \Psi(\xi)})\bigg] 
e^{i \xi \cdot x} e^{i\tau t} \widehat{f}_r(\tau) \widehat{g}_r(\xi).
\end{align*}
By inverse Fourier transform and \eqref{X:pdf}, we have
\begin{align*}
I = (2\pi)^{1+d}\bigg[f_r(0) g_r(0) - &f_r(t) (p_t \ast g_r)(0) \\
&- \sum_{j=1}^n a_j \Big( f_r(t-t^j) g_r(x-x^j) - 
f_r(t) (p_{t^j} \ast g_r)(x-x^j)\Big)\bigg].
\end{align*}
By the definition of $r$, for every $j$, either $|t-t^j| \ge r^\alpha$ or $|x-x^j| \ge r$, 
thus $f_r(t-t^j) g_r(x-x^j) = 0$. Moreover, since $t/r^\alpha \ge a/M^\alpha \ge \rho$, 
we have $f_r(t) = 0$ and hence
\begin{equation}\label{Eq:I}
I = (2\pi)^{1+d} r^{-\alpha-d}.
\end{equation}

On the other hand, by the Cauchy--Schwarz inequality and \eqref{Eq:var_integral},
\begin{align*}
I^2 \le \frac{1}{K_0} \,
\E\bigg[ \bigg( u(t, x) - \sum_{j=1}^n a_j u(t^j, x^j) \bigg)^2 \bigg] 
\int_{\R} \int_{\R^d} \big|\widehat{f}_r(\tau) \widehat{g}_r(\xi)\big|^2 
\big(C_\Psi^2|\xi|^{2\alpha}+|\tau|^2\big)|\tau|^{2H-1} |\xi|^{\beta} d\tau \, d\xi.
\end{align*}
Note that $\widehat{f}_r(\tau) = \widehat{f}(r^\alpha \tau)$ and 
$\widehat{g}_r(\xi) = \widehat{g}(r\xi)$. Then by scaling, we have
\begin{align*}
&\int_{\R} \int_{\R^d} \big|\widehat{f}_r(\tau) \widehat{g}_r(\xi)\big|^2 
\big(C_\Psi^2|\xi|^{2\alpha}+|\tau|^2\big)|\tau|^{2H-1} |\xi|^{\beta} d\tau \, d\xi\\
& = r^{-2\alpha -2\alpha H -\beta -d} \int_{\R} \int_{\R^d} 
\big|\widehat{f}(\tau) \widehat{g}(\xi)\big|^2 
\big(C_\Psi^2|\xi|^{2\alpha}+|\tau|^2\big)|\tau|^{2H-1} |\xi|^{\beta} d\tau \, d\xi\\
& =:  r^{-2\alpha -2\alpha H -\beta -d} C_0,
\end{align*}
where $C_0$ is a finite constant since $\widehat{f}$ and $\widehat{g}$ are 
rapidly decreasing functions. It follows that
\begin{equation}\label{Eq:I^2}
I^2 \le \frac{C_0}{K_0}\, r^{-2\alpha -2\alpha H -\beta -d} \,\,
\E\Bigg[ \bigg( u(t, x) - \sum_{j=1}^n a_j u(t^j, x^j) \bigg)^2 \Bigg].
\end{equation}
Combining \eqref{Eq:I} and \eqref{Eq:I^2}, we conclude that
\[ \E\Bigg[ \bigg( u(t, x) - \sum_{j=1}^n a_j u(t^j, x^j) \bigg)^2 \Bigg] 
\ge \frac{(2\pi)^{2+2d}K_0}{C_0}\, r^{2\theta_2}. \]
This completes the proof of Lemma \ref{Lem:SHE-a2}.
\end{proof}

Under conditions \eqref{Psi:cond}, \eqref{h:cond} and \eqref{theta:cond},
Lemmas \ref{Lem:DMX17}, \ref{Lem:SHE-a1} and \ref{Lem:SHE-a2} imply that 
for any compact rectangle $T$ in $(0, \infty) \times \R^d$, there exist positive finite 
constants $c_1$ and $c_3$ such that for all $(t, x), (s, y) \in T$,
\begin{equation}\label{SHE:DeltaBD}
c_3 \Delta((t, x), (s, y)) \le \|u(t, x) - u(s, y)\|_{L^2} \le c_1 \Delta((t, x), (s, y)).
\end{equation}
See also \cite[Theorem 4.1]{HS}.

By applying our results, we obtain the following theorem, which strengthens the 
regularity results in Propositions 3.7 and 3.10 of \cite{HSWX} and provides more 
precise information and bounds for the limiting constants.

\begin{theorem}\label{Thm:fcSHE}
Suppose $\Psi$ and $h$ satisfy conditions \eqref{Psi:cond} and \eqref{h:cond} respectively. 
Suppose $\theta_1$ and $\theta_2$ satisfy condition \eqref{theta:cond}. 
Then the following statements hold.
\medskip
\begin{enumerate}
\item[(i)] Chung-type law of the iterated logarithm:
For any fixed $(t_0, x_0) \in (0, \infty) \times \R^d$, 
\begin{equation*}
\liminf_{r \to 0+} \sup_{\substack{t > 0,\, x \in \R^d:\\ \Delta((t, x), (t_0, x_0)) \le r}} 
\frac{|u(t, x) - u(t_0, x_0)|}{r(\log\log(1/r))^{-1/Q}} = \kappa_1^{1/Q} \quad \text{a.s.},
\end{equation*}
where $Q = \frac{1}{\theta_1} + \frac{d}{\theta_2}$ and $\kappa_1$ is a positive finite 
constant given by Theorem \ref{Thm:ChungLIL}.
\medskip
\item[(ii)] The exact local modulus of continuity:
For any fixed $(t_0, x_0) \in (0, \infty) \times \R^d$,
\begin{equation*}
\qquad \lim_{r \to 0+} \sup_{\substack{t > 0,\, x \in \R^d:\\ 
0 < d((t, x), (t_0, x_0)) \le r}} 
\frac{|u(t, x) - u(t_0, x_0)|}
{d((t, x), (t_0, x_0))\sqrt{\log\log(d((t, x), (t_0, x_0))^{-1})}} = \sqrt 2
\quad \text{a.s.},
\end{equation*}
where $d((t, x), (t_0, x_0)) = \|u(t, x) - u(t_0, x_0)\|_{L^2}$, and
\begin{equation*}
\qquad \lim_{r \to 0+} \sup_{\substack{t > 0, \, x \in \R^d:\\ 
0 < \Delta((t, x), (t_0, x_0)) \le r}} 
\frac{|u(t, x) - u(t_0, x_0)|}
{\Delta((t, x), (t_0, x_0))\sqrt{\log\log(\Delta((t, x), (t_0, x_0))^{-1})}} = \kappa_2 
\quad \text{a.s.}
\end{equation*}
for some positive finite constant $\kappa_2$ such that
\begin{equation*}
\sqrt{2}\,c_3 \le \kappa_2 \le \sqrt{2}\, c_1,
\end{equation*}
where $c_1, c_3$ are constants satisfying \eqref{SHE:DeltaBD}, with $T$ being any 
neighborhood of $(t_0, x_0)$.
\medskip
\item[(iii)] The exact uniform modulus of continuity:
For any compact rectangle $T$ in $(0, \infty) \times \R^d$,
\begin{equation*}
\lim_{r \to 0+} \sup_{\substack{(t, x), (s, y) \in T:\\ 0 < d((t, x), (s, y)) \le r}} 
\frac{|u(t, x) - u(s, y)|}{d((t, x), (s, y))\sqrt{\log(d((t, x), (s, y))^{-1})}} = \kappa_3
\quad \text{a.s.}
\end{equation*}
and
\begin{equation*}
\lim_{r \to 0+} \sup_{\substack{(t, x), (s, y) \in T:\\ 0 < \Delta((t, x), (s, y)) \le r}} 
\frac{|u(t, x) - u(s, y)|}{\Delta((t, x), (s, y))\sqrt{\log(\Delta((t, x), (s, y))^{-1})}} = \kappa_4
\quad \text{a.s.}
\end{equation*}
for some positive finite constants $\kappa_3$, $\kappa_4$ satisfying
\begin{equation*}
\sqrt{2Q c_2} \, c_1^{-1} \le \kappa_3 \le \sqrt{2Q}
\quad \text{and} \quad \sqrt{2Q c_2} \le \kappa_4 \le \sqrt{2Q}\, c_1,
\end{equation*}
where $Q = \frac{1}{\theta_1} + \frac{d}{\theta_2}$, 
$c_1$ is the constant in \eqref{SHE:DeltaBD} and 
$c_2$ is the constant in Lemma \ref{Lem:SHE-a2}.
\end{enumerate}
\end{theorem}

\begin{proof}
By Lemma \ref{SHE:law}, (i)--(iii) hold if and only if they hold
for the Gaussian random field $v(t, x)$ defined in \eqref{Def:v}.
By Lemmas \ref{Lem:SHE-a1} and \ref{Lem:SHE-a2}, 
$v(t, x)$ satisfies Assumptions \ref{a1}, \ref{a2} and \ref{a3}.
Therefore, the desired results follows from Theorems \ref{Thm:ChungLIL},
\ref{Thm:LIL} and \ref{Thm:MC}.
\end{proof}

\begin{remark}
The constants $\kappa_1$ and $\kappa_2$ are independent of the point 
$(t_0, x_0)$. This is due to the decomposition \eqref{SHE:de} above and 
the fact that the random field $\{U(t, x), t \ge 0, x \in \R^k\}$ has stationary 
increments \cite{HSWX}.
\end{remark}

In the theorem below, we consider the special case that $\Psi(\xi) = |\xi|^\alpha$,
which corresponds to the equation \eqref{Eq:SHE} with $\mathscr L$ being the
fractional Laplacian $-(-\Delta)^{\alpha/2}$.
We are able to obtain the exact constants for 
the LIL in time variable and space variable respectively.
This result strengthens Corollaries 3.8 and 3.9 of \cite{HSWX}.

\begin{theorem}
Suppose $\Psi(\xi) = |\xi|^\alpha$ and $h(\xi) = |\xi|^{-\beta}$, where
$0 < \alpha \le 2$, $0 < \beta < d$.
Suppose $\theta_1$ and $\theta_2$ satisfy condition \eqref{theta:cond}.
Then, for any fixed $(t_0, x_0) \in (0, \infty) \times \R^d$, almost surely,
\begin{equation}\label{SHE:LILt}
\limsup_{\delta \to 0} \frac{|u(t_0+\delta, x_0) - u(t_0, x_0)|}
{|\delta|^{\theta_1}\sqrt{\log\log(1/|\delta|)}} = 
\left(2c_{H,d} \iint_{\R\times\R^d} |e^{-i\tau} - 1|^2
\frac{|\tau|^{1-2H}|\xi|^{-\beta}}{|\tau|^2 + |\xi|^{2\alpha}}\, d\tau\,d\xi\right)^{1/2}
\end{equation}
and
\begin{equation}\label{SHE:LILx}
\limsup_{|\eps| \to 0}
\frac{|u(t_0, x_0+\eps) - u(t_0, x_0)|}{|\eps|^{\theta_2}\sqrt{\log\log(1/|\eps|)}} = 
\left(2c_{H,d} \iint_{\R\times\R^d}
|e^{-i\xi_1} - 1|^2 \frac{|\tau|^{1-2H}|\xi|^{-\beta}}{|\tau|^2 + |\xi|^{2\alpha}}\,
d\tau\,d\xi\right)^{1/2}.
\end{equation}
\end{theorem}

\begin{proof}
Let $\kappa_5$ and $\kappa_6$ denote the quantity on the right-hand sides of
\eqref{SHE:LILt} and \eqref{SHE:LILx} respectively.
Fix $(t_0, x_0) \in (0, \infty) \times \R^d$.
We claim that
\begin{equation}\label{SHE:LILc1}
\left\|(u(t_0+s, x_0) - u(t_0, x_0)\right\|_{L^2}
= |s|^{\theta_1} \big(\kappa_5 + o(1)\big)
\quad \text{as } s \to 0,
\end{equation}
and
\begin{equation}\label{SHE:LILc2}
\left\|(u(t_0, x_0 + y) - u(t_0, x_0)\right\|_{L^2}
= |y|^{\theta_2} \big(\kappa_6 + o(1) \big)
\quad \text{as } |y| \to 0.
\end{equation}
Once the claims are proved to be true, we can consider a sequence of 
neighborhoods converging to the point $t_0$ and $x_0$ respectively, and apply 
Theorem \ref{Thm:LIL} to the processes $\{ u(t, x_0), t > 0 \}$ and
$\{ u(t_0, x), x \in \R^d \}$ to obtain \eqref{SHE:LILt} and \eqref{SHE:LILx} respectively.

To prove \eqref{SHE:LILc1}, for $s > 0$, we use \eqref{Eq:var_lin_comb} to get that
\begin{align*}
& \|u(t_0+s, x_0) - u(t_0, x_0)\|_{L^2}^2\\
& = c_{H,d} \iint_{\R\times\R^d}
\big|(e^{-i\tau (t_0+s)} - e^{-(t_0+s)|\xi|^\alpha}) 
- (e^{-i\tau t_0} - e^{- t_0|\xi|^\alpha})\big|^2\,
\frac{|\tau|^{1-2H} |\xi|^{-\beta}}{|\tau|^2 + |\xi|^{2\alpha}} \, d\tau\, d\xi\\
& = c_{H,d} \iint_{\R\times\R^d}
\big|e^{-i\tau t_0}(e^{-i\tau s} - 1) 
- e^{- t_0 |\xi|^\alpha}(e^{- s|\xi|^\alpha} - 1)\big|^2\,
\frac{|\tau|^{1-2H} |\xi|^{-\beta}}{|\tau|^2 + |\xi|^{2\alpha}}\, d\tau\, d\xi.
\end{align*}
Then the change of variables 
$\tau \mapsto s^{-1}\tau$ and $\xi \mapsto s^{-1/\alpha} \xi$ leads to
\begin{align*}
& \|u(t_0+s, x_0) - u(t_0, x_0)\|_{L^2}^2\\
& = c_{H,d} \, s^{2\theta_1} \iint_{\R\times\R^d}
\big|(e^{-i\tau} - 1) 
- e^{s^{-1} (i\tau t_0 - t_0 |\xi|^\alpha)}(e^{- |\xi|^\alpha} - 1)\big|^2\,
\frac{|\tau|^{1-2H} |\xi|^{-\beta}}{|\tau|^2 + |\xi|^{2\alpha}}\, d\tau\, d\xi.
\end{align*}
Similarly, for $s > 0$ small,
\begin{align*}
& \|u(t_0-s, x_0) - u(t_0, x_0)\|_{L^2}^2\\
& = c_{H,d} \,s^{2\theta_1} \iint_{\R\times\R^d}
\big|(e^{-i\tau} - 1) 
- e^{s^{-1} (i\tau (t_0-s) - (t_0-s) |\xi|^\alpha)}(e^{- |\xi|^\alpha} - 1)\big|^2\,
\frac{|\tau|^{1-2H} |\xi|^{-\beta}}{|\tau|^2 + |\xi|^{2\alpha}}\, d\tau\, d\xi.
\end{align*}
Since $0 \le 1 - e^{- |\xi|^\alpha}\le \min(1, |\xi|^\alpha)$ and
\begin{align*}
\iint_{\R\times\R^d} \min(1, |\xi|^{2\alpha})\,
\frac{|\tau|^{1-2H} |\xi|^{-\beta}}{|\tau|^2 + |\xi|^{2\alpha}}\, d\tau\, d\xi < \infty,
\end{align*}
by the dominated convergence theorem, we have
$|s|^{-2\theta_1}\|u(t_0+s, x_0) - u(t_0, x_0)\|_{L^2}^2 \to \kappa_5^2$ as $s \to 0$,
which is exactly \eqref{SHE:LILc1}.

For \eqref{SHE:LILc2}, we let $y \in \R^d \setminus \{0\}$ 
and use \eqref{Eq:var_lin_comb} again to get that
\begin{align*}
& \|u(t_0, x_0 + y) - u(t_0, x_0)\|_{L^2}^2\\
& = c_{H,d} \iint_{\R\times\R^d}
\big| e^{-i\xi \cdot (x_0 + y)} (e^{-i\tau t_0} - e^{-t_0|\xi|^\alpha})
- e^{-i\xi \cdot x_0} (e^{-i\tau t_0} - e^{-t_0|\xi|^\alpha})\big|^2\,
\frac{|\tau|^{1-2H} |\xi|^{-\beta}}{|\tau|^2 + |\xi|^{2\alpha}} \, d\tau\, d\xi\\
& = c_{H,d} \iint_{\R\times\R^d}
\big|(e^{-i\xi \cdot y} - 1)(1 - e^{i\tau t_0 - t_0|\xi|^\alpha})\big|^2\,
\frac{|\tau|^{1-2H} |\xi|^{-\beta}}{|\tau|^2 + |\xi|^{2\alpha}}\, d\tau\, d\xi.
\end{align*}
By the change of variables 
$\tau \mapsto |y|^{-\alpha}\tau$ and $\xi \mapsto |y|^{-1}\xi$, 
the above expression is equal to
\begin{align*}
c_{H,d}\, |y|^{2\theta_2} \iint_{\R\times\R^d}
\big|\big(e^{-i\xi \cdot \frac{y}{|y|}} - 1\big)
\big(1 - e^{|y|^{-\alpha}(i\tau t_0 -t_0|\xi|^\alpha)}\big)\big|^2\,
\frac{|\tau|^{1-2H} |\xi|^{-\beta}}{|\tau|^2 + |\xi|^{2\alpha}}\, d\tau\, d\xi.
\end{align*}
Then, by a rotation of the variable $\xi$ which takes the unit vector 
$\frac{y}{|y|}$ to the basis vector $\mathrm{\bf e}_1 = (1, 0, \dots, 0)$, 
we deduce that
\begin{align*}
& \|u(t_0, x_0 + y) - u(t_0, x_0)\|_{L^2}^2\\
& = c_{H,d}\, |y|^{2\theta_2} \iint_{\R\times\R^d}
\big|\big(e^{-i\xi_1} - 1\big)
\big(1 - e^{|y|^{-\alpha}(i\tau t_0 -t_0|\xi|^\alpha)}\big)\big|^2\,
\frac{|\tau|^{1-2H} |\xi|^{-\beta}}{|\tau|^2 + |\xi|^{2\alpha}}\, d\tau\, d\xi.
\end{align*}
By the dominated convergence theorem, we get \eqref{SHE:LILc2} as $y \to 0$.
The proof is complete.
\end{proof}

\section{Strongly LND anisotropic Gaussian fields with non-stationary increments}

Finally, we construct a class of anisotropic Gaussian random fields
that have strong LND property but do not have stationary increments. 
Let $f: \R^k \to \R$ be a nonnegative function such that for all $\xi \in \R^k$,
\begin{equation}\label{eg2:f}
\frac{C_1}{(\sum_{j=1}^k |\xi_j|^{\alpha_j})^{Q+2}}\le
f(\xi) \le \frac{C_2}{(\sum_{j=1}^k |\xi_j|^{\alpha_j})^{Q+2}},
\end{equation}
where $0 < \alpha_j < 1$, $Q = \sum_{j=1}^k \alpha_j^{-1}$ and
$C_1, C_2$ are positive finite constants. It can be verified that $f$ satisfies
\[
\int_{\R^k} \min\{1, |\xi|^2\} f(\xi) d\xi < \infty.
\]
Define the Gaussian random field $v = \{v(x), x \in \R^k\}$ by
\begin{equation}\label{eg2}
v(x) = \int_{\R^k} \prod_{j=1}^k (e^{ix_j \xi_j} - 1) W(d\xi),
\end{equation}
where $W$ is a centered complex-valued Gaussian random measure whose control 
measure has density $f$, meaning that for all Borel sets $A, B \subset \R^k$,
\[ \E[W(A)\overline{W(B)}] = \int_{A \cap B} f(\xi) d\xi, \quad \text{and} \quad
W(-A) = \overline{W(A)}. \]
This implies that $v$ is real-valued.
Note that $v$ does not have stationary increments.
Still, we can verify that $v$ satisfies Assumptions \ref{a1} and \ref{a2} in Lemmas 
\ref{Lem:8.1} and \ref{Lem:8.2} below.

\begin{lemma}\label{Lem:8.1}
Let $T$ be a compact rectangle in $\R^k$.
Then the process $\{ v(A, x), A \in \mathscr{B}(\R_+), x \in T\}$ 
defined by
\[ v(A, x) = \int_{\{\max_j|\xi_j|^{\alpha_j} \in A\}} 
\prod_{j=1}^k (e^{ix_j \xi_j} - 1) W(d\xi) \]
satisfies Assumption \ref{a1}(a). Moreover, there exists a finite constant $c_0$
such that for all $0 \le a < b \le \infty$ and all $x, y \in T$,
\begin{equation}\label{eg2:approx}
\|v(x) - v([a, b), x) - v(y) + v([a, b), y)\|_{L^2}
\le c_0 \bigg( \sum_{j=1}^k a^{\gamma_j} |x_j - y_j| + b^{-1}\bigg),
\end{equation}
where $\gamma_j = \alpha_j^{-1}-1$.
In particular, Assumption \ref{a1}(b) is satisfied for $a_0 = 0$.
\end{lemma}

\begin{proof}
It is clear that $v(A, x)$ satisfies Assumption \ref{a1}(a).
For \eqref{eg2:approx}, by writing $v(x) - v(y)$ as the telescoping sum
\begin{align*}
[v(x_1, \dots, x_k) - v(y_1, x_2, \dots, x_k)] &+ [v(y_1, x_2, \dots, x_k) - v(y_1, y_2, x_3, \dots, x_k)]\\
&+ \cdots + [v(y_1, \dots, y_{k-1}, x_k) - v(y_1, \dots, y_k)]
\end{align*}
and similarly for $v([a, b), x) - v([a, b), y)$, it is enough to prove \eqref{eg2:approx}
for $x$ and $y$ that only differ in one coordinate.
By re-arranging coordinates, we only need to consider the case that 
$x = (x_1, \dots, x_k)$ and $y = (y_1, x_2, \dots, x_k)$.
Note that 
\begin{align*}
&v(x) - v([a, b), x) - v(y) + v([a, b), y)\\
&=[v([0, a), x) - v([0, a), y)] + [v([b, \infty), x) - v([b, \infty), y)].
\end{align*}
We estimate the two terms separately.
By $|e^{iz}-e^{iz'}| \le |z-z'|$, $|e^{iz} - 1| \le 2$ and \eqref{eg2:f},
\begin{align*}
&\|v([0, a), x) - v([0, a), y)\|_{L^2}^2\\
&= \int_{\{\max_j |\xi_j|^{\alpha_j} < a\}} 
\bigg| \prod_{j=1}^k (e^{ix_j\xi_j}-1) - 
(e^{iy_1\xi_1} - 1)\prod_{j=2}^k (e^{ix_j\xi_j}-1) \bigg|^2 f(\xi) \, d\xi\\
& \le 2^{2k-2} C_2 |x_1-y_1|^2 \int_{\{\max_j |\xi_j|^{\alpha_j} < a\}} 
\frac{|\xi_1|^2}{(\sum_{j=1}^k |\xi_j|^{\alpha_j})^{Q+2}}\, d\xi.
\end{align*}
Note that $|\xi_1|^2 = (|\xi_1|^{\alpha_1})^{2+{2(1-\alpha_1)}/{\alpha_1}} 
\le (\sum_{j=1}^k|\xi_j|^{\alpha_j})^{2 + {2(1-\alpha_1)}/{\alpha_1}}$.
Then, by the change of variables $\xi_j \mapsto a^{\alpha_j^{-1}}\xi_j$, 
followed by another change $\xi_j \mapsto z_j^{2/\alpha_j}$,
the last integral is equal to
\begin{align*}
&a^{2\alpha_1^{-1}-2} \int_{\{\max_j |\xi_j|^{\alpha_j} < 1\}}
\Big(\sum_{j=1}^k|\xi_j|^{\alpha_j}\Big)^{-Q+ {2(1-\alpha_1)}/{\alpha_1}} d\xi\\
& \le  a^{2\alpha_1^{-1}-2}\prod_{j=1}^k(2/\alpha_j) \int_{\{\max_j |z_j|^2 < 1\}} 
|z|^{-k+{4(1-\alpha_1)}/{\alpha_1}} dz\\
& \le C a^{2\alpha_1^{-1}-2},
\end{align*}
where $C$ is a finite constant, so we get the estimate
$\|v([0, a), x) - v([0, a), y)\|_{L^2}^2\le C a^{2\gamma_j} |x_1-y_1|^2$.
On the other hand, by $|e^{iz} - 1| \le 2$ and \eqref{eg2:f},
\begin{align*}
\|v([b, \infty), x) - v([b, \infty), y)\|_{L^2}^2
\le 2^{2k} C_2 \int_{\{\max_j |\xi_j|^{\alpha_j} \ge b\}} 
\frac{1}{(\sum_{j=1}^k |\xi_j|^{\alpha_j})^{Q+2}}\, d\xi.
\end{align*}
Now, by similar changes of variables $\xi_j \mapsto b^{\alpha_j^{-1}}\xi_j$
and then $\xi_j \mapsto z_j^{2/\alpha_j}$, we get that 
$\|v([b, \infty), x) - v([b, \infty), y)\|_{L^2}^2 \le C b^{-2}$.
Combining the two estimates above finishes the proof of \eqref{eg2:approx}.
\end{proof}

\begin{lemma}\label{Lem:8.2}
Define $\Delta(x, y) = \sum_{j=1}^k |x_j - y_j|^{\alpha_j}$.
Let $T$ be a compact rectangle in $\R^k$ away from the axes.
Then, there exists a positive finite constant $c_2$ such that 
for all $n \ge 1$, for all $x, x^1, \dots, x^n \in T$,
\[ \mathrm{Var}(v(x)|v(x^1), \dots, v(x^n)) \ge 
c_2 \min_{1\le \ell \le n} \Delta^2(x, x^\ell). \]
In particular, this implies that $\|v(x) - v(y)\|_{L^2} \ge \sqrt{c_2} \Delta(x, y)$ 
for all $x, y \in T$.
\end{lemma}

\begin{proof}
We may assume that $a \le |x_j| \le b$ for all $x = (x_1, \dots, x_k) \in T$,
where $0 < a < 1 < b < \infty$ are constants.
It suffices to prove that there exists a positive finite constant $c$ 
such that for all $n \ge 1$, for all $x, x^1, \dots, x^n \in T$ and 
all $a_1, \dots, a_n \in \R$,
\begin{equation}\label{eg2:LND}
\E\bigg[ \Big( v(x) - \sum_{\ell=1}^n a_\ell v(x^\ell) \Big)^2 \bigg] \ge c r^2,
\quad \text{where }r = \min_{1\le \ell \le n} \max_{1\le j\le k} |x_j - x^\ell_j|^{\alpha_j}.
\end{equation}
By \eqref{eg2:f},
\begin{equation}\label{eg2:varLB}
\begin{split}
&\E\bigg[ \Big( v(x) - \sum_{\ell=1}^n a_\ell v(x^\ell) \Big)^2 \bigg]\\
&\ge C_1 \int_{\R^k} \bigg| \prod_{j=1}^k (e^{ix_j\xi_j} - 1) - \sum_{\ell=1}^n
a_\ell \prod_{j=1}^k (e^{ix^\ell_j \xi_j} - 1) \bigg|^2 
\frac{d\xi}{(\sum_{j=1}^k |\xi_j|^{\alpha_j})^{Q+2}}.
\end{split}
\end{equation}
Let $\rho = \min\{1, a(2b)^{-\alpha^*/\alpha_*}\}$,
where $\alpha^* = \max\{\alpha_1, \dots, \alpha_k\}$ and 
$\alpha_* = \min\{\alpha_1, \dots, \alpha_k\}$.
For each $j = 1, \dots, k$, let $\phi^j: \R \to \R_+$ be a nonnegative smooth
function supported on $[-\rho, \rho]$ satisfying $\phi^j(0) = 1$.
Let $\phi^j_r(z) = r^{-\alpha_j^{-1}}\phi^j(r^{-\alpha_j^{-1}}z)$
and let $\widehat \phi^j_r$ denote the Fourier transform of $\phi^j_r$.
Consider the integral
\[ 
I:= \int_{\R^k} \bigg[\prod_{j=1}^k (e^{ix_j\xi_j} - 1) - \sum_{\ell=1}^n
a_\ell \prod_{j=1}^k (e^{ix^\ell_j \xi_j} - 1)\bigg] \prod_{j=1}^k \big[e^{-ix_j\xi_j} 
\widehat\phi^j_r(\xi_j)\big] \, d\xi. 
\]
Then, by inverse Fourier transform,
\[ I = (2\pi)^k \bigg[ \prod_{j=1}^k (\phi^j_r(0) - \phi^j_r(x_j) ) -\sum_{\ell=1}^n a_\ell
\prod_{j=1}^k (\phi^j_r(x_j - x_j^\ell) - \phi^j_r(x_j))\bigg]. \]
Note that $\phi^j_r(0) = r^{-Q}$. 
Since $r \le (2b)^{\alpha^*}$, we have
$r^{-\alpha_j^{-1}}|x_j| \ge a(2b)^{-\alpha^*/\alpha_*} \ge \rho$, 
thus $\phi^j_r(x_j) = 0$.
Also, for each $\ell$, by the definition of $r$, $\max_j |x_j - x^\ell_j|^{\alpha_j} \ge r$, 
so there exists some $j$ such that $|x_j - x^\ell_j|^{\alpha_j} \ge r$.
For this $j$, $r^{-\alpha_j^{-1}}|x_j - x^\ell_j| \ge 1 \ge \rho$, and hence
$\phi^j_r(x_j - x^\ell_j) = 0$.
This implies that for each $\ell$,
\[ \prod_{j=1}^k (\phi^j_r(x_j - x_j^\ell) - \phi^j_r(x_j)) = 0. \]
Therefore, we have
\begin{equation}\label{eg2:I}
I = (2\pi)^k r^{-Q}.
\end{equation}

On the other hand, by the Cauchy--Schwarz inequality and \eqref{eg2:varLB} above,
\[ I^2 \le \frac{1}{C_1} \E\bigg[ \Big( v(x) - \sum_{\ell=1}^n a_\ell v(x^\ell) \Big)^2 \bigg] 
\times \int_{\R^k} \prod_{j=1}^k |\widehat \phi^j_r(\xi_j)|^2 
\Big(\sum_{j=1}^k |\xi_j|^{\alpha_j}\Big)^{Q+2} d\xi. \]
By $\widehat \phi^j_r(\xi_j) = \widehat \phi^j(r^{\alpha_j^{-1}}\xi_j)$ and the change of 
variables $\xi_j \mapsto r^{-\alpha_j^{-1}}\xi_j$, the integral on the right-hand side is 
equal to
\begin{align*}
r^{-2Q-2} \int_{\R^k} \prod_{j=1}^k |\widehat \phi^j(\xi_j)|^2 
\Big(\sum_{j=1}^k |\xi_j|^{\alpha_j}\Big)^{Q+2} d\xi = C_0 r^{-2Q-2}.
\end{align*}
Therefore, together with \eqref{eg2:I}, we have
\[ (2\pi)^{2k} r^{-2Q} = I^2 \le \frac{C_0}{C_1} r^{-2Q-2} \,
\E\bigg[ \Big( v(x) - \sum_{\ell=1}^n a_\ell v(x^\ell) \Big)^2 \bigg] \]
and \eqref{eg2:LND} follows. The proof is complete.
\end{proof}

\begin{corollary}
Let $T$ be a compact rectangle in $\R^k$ away from the axes.
Then, Theorems \ref{Thm:ChungLIL}, \ref{Thm:LIL} and \ref{Thm:MC} 
can be applied to the Gaussian random field $v$ defined by \eqref{eg2}.
\end{corollary}

We point out that even though the Gaussian random field $v = \{v(x), x \in \R^k\}$ 
defined by (\ref{eg2}) and the fractional Brownian sheet with parameters 
$(\alpha_1, \ldots, \alpha_k) \in (0, 1)^k$ share some similarity in their definitions
and many sample path properties, some of their other fine properties such as 
Chung's LILs and exact Hausdorff measure functions are rather different (see Lee \cite{L21}).
Instead, it can be proved that these latter properties of $v = \{v(x), x \in \R^k\}$
are similar to those in \cite{LX10, LX12} for Gaussian random fields with 
stationary increments and spectral density $f$ that satisfies (\ref{eg2:f}). 
%It would be interesting to 

 \medskip
\bigskip{\bf Acknowledgements}\, The research of Y. Xiao is partially supported 
by NSF grant DMS-1855185.

\end{document}